\documentclass[reqno,english]{amsart}
\usepackage{amsfonts,amsmath,latexsym,verbatim,amscd,mathrsfs,color,array}

\usepackage{amsmath,amssymb,amsthm,amsfonts,graphicx,color}
\usepackage{amssymb}
\IfFileExists{pdfsync.sty}{\usepackage{pdfsync}}{}
\ifdefined\synctex\synctex=1\fi
\usepackage{epstopdf}
\usepackage[colorlinks=true,hypertexnames=false]{hyperref}
\usepackage{subfigure}
\usepackage{bm}
\usepackage{needspace}
\usepackage{booktabs} 
\usepackage{amsmath} 

\makeatletter
\def\@currentlabel{2.1}\label{e:dispaa}
\def\@currentlabel{2.21}\label{e:dispau}
\def\@currentlabel{2.22}\label{e:dispav}
\def\@currentlabel{2.23}\label{e:dispaw}
\def\@currentlabel{2.24}\label{e:dispax}
\def\theequation{\thesection.\@arabic\c@equation}
\makeatother
\let\oldbibliography\thebibliography
\renewcommand{\thebibliography}[1]{%
\oldbibliography{#1}%
\setlength{\itemsep}{0pt}%
}

\renewcommand{\theequation}{\thesection.\arabic{equation}}
\newtheorem{lemma}{Lemma}[section]

\newtheorem{definition}{Definition}[section]

\newtheorem{proposition}{Proposition}[section]
\newtheorem{corollary}{Corollary}[section]
\newtheorem{remark}{Remark}[section]
\newcommand{\bremark}{\begin{remark} \em}
\newcommand{\eremark}{\end{remark} }

\newtheorem*{thmA}{Theorem A}
\newtheorem*{thmB}{Theorem B}

\newtheorem*{thmA'}{Theorem A'}
\newtheorem{open problem}{Open Problem}[section]
\newtheorem{open question}{Open Question}[section]

\newtheorem*{conjecture a}{Conjecture A}
\newtheorem{open questions on hexagonal lattice}{Open Questions on hexagonal lattice}[section]

\newtheorem{theorem}{Theorem}[section]

\newtheorem*{numerical conjecture}{Numerical Conjecture}
\newtheorem*{numerical assertion}{Numerical Assertion}

\newcommand{\R}{{\mathbb R}}

\newcommand{\BE}{\begin{equation}}
\newcommand{\BEN}{\begin{equation*}}
\newcommand{\EE}{\end{equation}}
\newcommand{\EEN}{\end{equation*}}
\newcommand{\BL}{\begin{lemma}}
\newcommand{\EL}{\end{lemma}}
\newcommand{\BT}{\begin{theorem}}
\newcommand{\ET}{\end{theorem}}
\newcommand{\BP}{\begin{proposition}}
\newcommand{\EP}{\end{proposition}}
\newcommand{\BC}{\begin{corollary}}
\newcommand{\EC}{\end{corollary}}

\renewcommand{\Im}{\operatorname{Im}}
\usepackage{amsmath}

\DeclareMathOperator*{\Minima}{minima}

\begin{document}

\title[Mueller-Ho Conjecture]{\bf A Complete Proof of Mueller--Ho Conjecture}

\author{Senping Luo}

\author{Juncheng Wei}

\address[S.~Luo]{School of Mathematics and statistics, Jiangxi Normal University, Nanchang, 330022, China}
\address[J.~Wei]{Department of Mathematics, Chinese University of Hong Kong, Shatin, NT, Hong
Kong}

\email[S.~Luo]{luosp1989@163.com}

\email[J.~Wei]{wei@math.cuhk.edu.hk}

\begin{abstract}
For $s>0$, and $z\in\mathbb H=\{z=x+iy:y>0\}$, define
 $\theta(s;z)=\sum_{m,n\in\mathbb Z}e^{-\pi s\big(m^2y+\frac{(mx-n)^2}{y}\big)}$ and $
 J(z;a,b)=\sum_{m,n\in\mathbb Z}e^{-\pi \big(m^2y+\frac{(mx-n)^2}{y}\big)}
                  \cos\bigl(2\pi(ma+nb)\bigr)$ for the classical and shifted theta functions (Gaussian lattice sums), respectively.

We prove that, up to the modular group, there exist three thresholds $0<\alpha_a<\alpha_b<\alpha_c<1$ such that
\begin{equation}\aligned\nonumber
\Minima_{z\in\mathbb H,\,(a,b)\in\mathbb R^2}\Big(\theta(1;z)+\alpha J(z;a,b)\Big)=
\begin{cases}
\;\big(e^{i{\pi}/{3}};1/3,1/3\big), &\hbox{if}\;\; \alpha\in[0,\alpha_a),\\
\;\big(e^{i{\pi}/{3}};1/3,1/3\big)\;\hbox{or}\;\big(e^{i{\varphi_{\alpha_a}}};1/2,1/2\big), &\hbox{if}\;\; \alpha=\alpha_a,\\
\;\big(e^{i{\varphi_\alpha}};1/2,1/2\big),\;\varphi_\alpha\in(\varphi_{\alpha_a},\frac{\pi}{2}),\;\; &\hbox{if}\;\; \alpha\in(\alpha_a,\alpha_b),\\
\;\big(i;1/2,1/2\big),\; &\hbox{if}\;\; \alpha\in[\alpha_b,\alpha_c],\\
\;\big(iy_\alpha;1/2,1/2\big),\; y_\alpha\in(1,\sqrt3],\;\;y_1=\sqrt3,\;\;&\hbox{if}\;\; \alpha\in(\alpha_c,1],
\end{cases}
\endaligned\end{equation}
where the complex variable $z$ and vector $(a,b)$, represent the total vortex shapes and relative positions of the two-component Bose gas, respectively. Consequently,
this completely proves the conjecture of Mueller and Ho~\cite{Mue2002} (2002) arising in Bose--Einstein condensates.

\end{abstract}

\maketitle

\section{Introduction and main results}\label{sec:introduction}

\subsection{Statement of the main results}

The theta function associated with the lattice $ \Lambda$ is defined as
 \begin{equation}\nonumber
\theta (s; \Lambda):=\sum_{\mathbb{P}\in\Lambda} e^{-\pi s|\mathbb{P}|^2}.
\end{equation}
Let $ z\in \mathbb{H}:=\{z\in\mathbb{C}: \Im(z)>0\}$ and
\begin{equation}\label{Lam}
\Lambda ={\frac{1}{\sqrt{\Im(z)}}}\Big({\mathbb Z}\oplus z{\mathbb Z}\Big)
\end{equation}
 with area of unit cell is $1$ be the lattice in $ \mathbb{R}^2$ parameterized by $z$, then one has
 \begin{equation}
 \label{thetas}
\theta (s; z):=\theta (s; \Lambda)=\sum_{(m,n)\in\mathbb{Z}^2} e^{- \frac{\pi\cdot s }{\Im(z) }|mz+n|^2}.
\end{equation}

Montgomery \cite{Mon1988} proved the following celebrated theorem:
\begin{thmA}[Montgomery, 1988] For $s>0$, up to the action by the modular group,
$$
\Minima_{z\in\mathbb{H}}\theta(s;z)=e^{i\frac{\pi}{3}}.
$$
\end{thmA}
Here $e^{i\pi/3}$ represents the triangular/hexagonal lattice. Theorem A originally comes from analytical number theory,
while has laid foundations in many lattice minimization problems and has been frequently used in applied mathematical and physical models such as  crystallizations of particle interactions (Radin \cite{Radin1987,Radin1991}, Blanc-Lewin \cite{Bla2015}, B\'etermin \cite{Bet2016,Bet2018}, B\'etermin-Zhang \cite{Bet2015}),  Ginzburg-Landau theory in superconductors (Abrikosov \cite{Abr}, Sandier-Serfaty \cite{Serfaty2010,Serfaty2012}, Serfaty \cite{Serfaty2014}), Ohta-Kawasaki models in di-block copolymers (Chen-Oshita \cite{Che2007}, Goldman-Muratov-Serfaty \cite{GMS2013}, Ren-Wei \cite{RW2015}), frame operator norms (Faulhuber \cite{Fau2018}) and many others. The related minimization of theta and eta functions on lattices has application to Gross-Pitaevskii theory in superfluids or  Bose-Einstein condensates, Ohta-Kawasaki models triblock copolymers (Luo-Ren-Wei \cite{Luo2019}) and many others.

For the higher dimensional cases, the corresponding minimization problems on lattices was first investigated by Sarnak and Strombergsson \cite{Sar2006} (in dimensions 3,4,8 and 24) and recently by Cohn-Kumar-Miller-Radchenko-Viazovska \cite{Coh2017, Coh2019} for universal optimality of $E_8$ and Leech lattices in dimensions 8 and 24. For relations with sphere packing problems, see Viazovska \cite{Via2017} (in dimension 8), Cohn-Kumar-Miller-Radchenko-Viazovska \cite{Coh2017} (in dimension 24), and Conway-Sloane \cite{Conwaybook} and the references therein. We also note that minimization problems for Dedekind eta function (equivalent to the theta function \eqref{thetas} via Mellin transform) also arise in the extremal determinants of Laplace-Beltrami Operators. See Osgood-Phillips-Sarnak \cite{Osg1988}, Faulhuber \cite{Fau2020}, B\'etermin-Sandier \cite{BetS}, and the reference therein.

The shifted theta function by a vector $u$ (Gaussian lattice sum) is defined as
\begin{equation}
 \label{thetasB}
J (s; u):=\theta (s; \Lambda+u)=\sum_{\mathbb{P}\in\Lambda} e^{-\pi s|\mathbb{P}+u|^2}.
\end{equation}

Let $u=(a,b)\in\mathbb R^2$, and by \eqref{Lam},
\begin{equation}\nonumber
 \theta (s; z,a,b):=\theta (s; \Lambda+u)=\sum_{m,n\in\mathbb Z}e^{-\pi \big(m^2y+\frac{(mx-n)^2}{y}\big)}
                  \cos\bigl(2\pi(ma+nb)\bigr).
\end{equation}

A recent important result proved that

\begin{thmB}[B\'etermin, Faulhuber, and Steinerberger \cite{BFS}] For $s>0$, up to the action by the modular group,
\begin{equation}\nonumber
 \max_{z\in\mathbb{H}}\min_{(a,b)\in\mathbb R^2}\theta (s; z,a,b)\;\;\hbox{is achived at}\;\; e^{i\pi/3}.
\end{equation}
\end{thmB}

In this perspective, since $\max_{(a,b)\in\mathbb R^2}\theta (s; z,a,b)=\theta (s; z)$, Montgomery's theorem can be reformulated as follows:
\begin{thmA'}[=Theorem A] For $s>0$, up to the action by the modular group,
\begin{equation}\nonumber
 \min_{z\in\mathbb{H}}\max_{(a,b)\in\mathbb R^2}\theta (s; z,a,b)\;\;\hbox{is achived at}\;\; e^{i\pi/3}.
\end{equation}
\end{thmA'}

Theorem B has deep applications in Gaussian Gabor systems in communications, the heat kernel problem, and the Landau constant; see B\'etermin, Faulhuber, and Steinerberger \cite{BFS} for a broad overview.

We now focus the another side connected to Theorem B.
Using the notations of Mueller-Ho, the Abrikosov energy is
\begin{equation}\nonumber
 J(z;a,b):=\theta (1; z,a,b).
\end{equation}
Throughout the paper, we write the two-component energy of the Bose gas through the interaction strength $\alpha$ as
\begin{equation}\nonumber
 E_\alpha(z;a,b):=\theta(1;z)+\alpha J(z;a,b),\;\; \alpha\in[0,1].
\end{equation}

Following Mueller and Ho, define
\begin{equation*}
 J(z;a,b):=\theta(1;z,a,b)
 =\sum_{m,n\in\mathbb Z}
 e^{-\pi\bigl(m^2y+\frac{(mx-n)^2}{y}\bigr)}
 \cos\bigl(2\pi(ma+nb)\bigr).
\end{equation*}
The two-component Bose-Einstein energy with interaction strength $\alpha$ is
\begin{equation*}
 E_\alpha(z;a,b):=\theta(1;z)+\alpha J(z;a,b),\;\; \alpha\in[0,1].
\end{equation*}

The problem of finding optimal lattice shapes arises in many physical models. Besides those examples we mentioned above, another example is the so-called vortices in Bose-Einstein condensates.
 Vortices in Bose-Einstein condensates are also called topological
defects, correspond to a zero of the order parameter with a circulation
of the phase. When they get numerous, these vortices arrange themselves on a lattice. In fact, in rotating Bose-Einstein condensates (BEC), vortices were first observed in two-component BEC's  (Matthews et al. \cite{Mat1999}): it is observed experimentally that the shape of the lattice can be either hexagonal or square depending on the rotational velocity of the condensate. Since then, following the pioneering work of Mueller-Ho \cite{Mue2002}, many authors have investigated the lattice shape in two component BEC's and for instance  Kasamatsu et al. \cite{KTU1,KTU2};  related works include Keeli--Oktel \cite{KO} who numerically calculate the elastic coefficients of the lattice. In Kuokanportti et al. \cite{Kuop}, the authors investigate the case of different masses and attractive interactions.

The ground state of a two-component condensate is well described by a Gross-Pitaevskii energy depending on the wave functions of each component which are coupled by an interaction term.
The construction of the Bose-Einstein condensates with large number of vortices was deduced in Ho \cite{Ho2001} (one-component case) and Mueller-Ho \cite{Mue2002} (two-component case)
Abrikosov~\cite{Abr} described vortex lattices in type-II
superconductors. Matthews, Anderson, Haljan, Hall, Wieman, and
Cornell~\cite{Mat1999} created vortices experimentally in
two-component Bose--Einstein condensates.
Ho~\cite{Ho2001} studied the many-vortex quantum Hall regime, and
Mueller and Ho~\cite{Mue2002} derived the two-component lattice energy.
Kasamatsu, Tsubota, and Ueda~\cite{Kas2003} studied symmetry changes
through Gross--Pitaevskii simulations; Ke\c{c}eli and Oktel~\cite{Kea2006}
analyzed Tkachenko modes, and Aftalion, Mason, and Wei~\cite{AMW}
derived a Thomas--Fermi point energy. See Kasamatsu, Tsubota, and
Ueda~\cite{Kas2005} for a review.
Kuopanportti, Huhtam\"aki, and M\"ott\"onen~\cite{Kuop} found further
lattice geometries for unequal masses and attractive intercomponent
coupling.

We completely classify the two-component vortex lattice shapes conjectured by Mueller-Ho \cite{Mue2002}.
We state it in a purely mathematical form, fostering the structure inside, as follows:
\begin{theorem}[Exact solutions to sum of two energy functionals; complete phase transitions and displacements]
\label{thm:classification} Up to the action by the modular group,
there are three thresholds $0<\alpha_a<\alpha_b<\alpha_c<1$ such that
\begin{equation}\aligned\nonumber
\Minima_{z\in\mathbb H,\,(a,b)\in\mathbb R^2}\Big(\theta(1;z)+\alpha J(z;a,b)\Big)=
\begin{cases}
\;\big(e^{i{\pi}/{3}};1/3,1/3\big), &\hbox{if}\;\; \alpha\in[0,\alpha_a),\\
\;\big(e^{i{\pi}/{3}};1/3,1/3\big)\;\hbox{or}\;\big(e^{i{\varphi_{\alpha_a}}};1/2,1/2\big), &\hbox{if}\;\; \alpha=\alpha_a,\\
\;\big(e^{i{\varphi_\alpha}};1/2,1/2\big),\;\varphi_\alpha\in(\varphi_{\alpha_a},\frac{\pi}{2}),\;\; &\hbox{if}\;\; \alpha\in(\alpha_a,\alpha_b),\\
\;\big(i;1/2,1/2\big),\; &\hbox{if}\;\; \alpha\in[\alpha_b,\alpha_c],\\
\;\big(iy_\alpha;1/2,1/2\big),\; y_\alpha\in(1,\sqrt3],\;\;y_1=\sqrt3,\;\;&\hbox{if}\;\; \alpha\in(\alpha_c,1],
\end{cases}
\endaligned\end{equation}
More precisely, all minimizing configurations are classified as follows:
\begin{enumerate}
\item If $\alpha\in(0,\alpha_a)$, the lattice is triangular and
      $(a,b)=(1/3,1/3)$.
\item If $\alpha=\alpha_a$, the transition
changes both the displacement and the lattice shape discontinuously. It admits two different type minimizers. They are triangular lattice with $(a,b)=(1/3,1/3)$
       and the centered rhombic configuration
      $(e^{i\varphi_{\alpha_a}};1/2,1/2)$.
\item If $\alpha\in(\alpha_a,\alpha_b)$, the vortex is rhombic changing with $\alpha$ and the relative position is $(a,b)=(1/2,1/2)$ keeping fixed.

\item If $\alpha\in[\alpha_b,\alpha_c]$, the vortex is square and the relative position is $(a,b)=(1/2,1/2)$, both are unchanged with the parameter.

\item If $\alpha\in(\alpha_c,1]$, the vortex is rectangular changing with the parameter $\alpha$ and the relative position is $(a,b)=(1/2,1/2)$. Furthermore $y_\alpha\in(1,\sqrt3]$ and $y_1=\sqrt3$.
\end{enumerate}

Here $\alpha_b$ and $\alpha_c$ are given by
\begin{equation}\label{eq:centered-thresholds}
 \alpha_b=-\frac{\theta_{xx}(1;i)}
                   {J_{xx}(i;1/2,1/2)},\qquad
 \alpha_c=-\frac{\theta_{yy}(1;i)}
                   {J_{yy}(i;1/2,1/2)}.
\end{equation}
The pair $(\alpha_a,\varphi_{\alpha_a})$ is determined uniquely by the minimization identity:
\[
\begin{aligned}
E_{\alpha}(e^{i\pi/3};1/3,1/3)
 =\min_{\varphi\in[\pi/3,\pi/2]}E_{\alpha}(e^{i\varphi};1/2,1/2),\;\;\alpha\in(0,3/{10}).
\end{aligned}
\]
$\varphi_\alpha\in(\pi/3,\pi/2)$
is the unique solution of
\begin{equation}\label{eq:rhombic-branch}
 \frac{d}{d\varphi}E_\alpha(e^{i\varphi};1/2,1/2)=0
 \;\;\hbox{for}\;\;\alpha\in(0,\alpha_b).
\end{equation}
$y_\alpha\in(1,\sqrt3]$ is
the unique solution of
\begin{equation}\label{eq:rectangular-branch}
  \frac{d}{d y}E_\alpha(iy;1/2,1/2)=0,\;\;\hbox{for}\;\;\alpha\in(\alpha_b,1).
\end{equation}

Numerically,
\begin{equation}\nonumber
\begin{aligned}
 \alpha_a=&0.172356\cdots,\;
 \alpha_b=0.373215\cdots,\;
 \alpha_c=0.925649\cdots,\;
 \varphi_{\alpha_a}=1.185955\cdots.
\end{aligned}
\end{equation}

\end{theorem}
Theorem \ref{thm:classification} provides a new pattern in lattice minimization problems both in lattice shapes and relative displacements.

For $0\leq\alpha\leq1$, we compare the triangular deep-hole energy
with the centered minimum,
\begin{equation}\label{eq:competing-energies}
 E_\alpha(e^{i\pi/3};1/3,1/3)
 \quad\text{and}\quad
 \min_{\tau\in\mathbb H}E_\alpha(\tau;1/2,1/2).
\end{equation}
The two deep holes of the triangular lattice correspond to
$\pm(1/3,1/3)$. The centered identity \eqref{eq:centered-identity},
proved in Section~\ref{sec:preliminaries}, identifies the centered
minimum with the minimum of the sum of theta functions classified by
our previous work~\cite[Theorem~2.1]{LW}.

\subsection{Outline of the proof and organization}

We write $D^2f$ for the real Hessian. For a real symmetric
$d\times d$ matrix $M$, $M\succ0$ means
$v^{\mathsf T}Mv>0$ for every $v\in\mathbb R^d\setminus\{0\}$.
We write $M\ge cI_d$ when
$v^{\mathsf T}Mv\ge c\|v\|_2^2$ for all $v\in\mathbb R^d$,
where $I_d$ is the identity matrix and
$\|v\|_2^2=\sum_{j=1}^d v_j^2$.

The proof combines a minimum principle in $(x,y,a,b)$ with modular
reduction (Proposition~\ref{Prop21}) and compactness
(Lemma~\ref{lem:cusp}). A minimizing representative satisfies
\[
 0\leq x\leq\frac12,\qquad x^2+y^2\geq1,\qquad
 \frac{\sqrt3}{2}\leq y<2,\qquad (a,b)\in\mathbb R^2/\mathbb Z^2.
\]

The principal comparison near the first transition is
\[
 \min_{z\in\mathbb H,\,(a,b)\in\mathbb R^2}E_\alpha(z;a,b)
 =\min\left\{E_\alpha(e^{i\pi/3};1/3,1/3),\,
       \min_{\tau\in\mathbb H}E_\alpha(\tau;1/2,1/2)\right\},
 \qquad \frac{1723}{10000}\leq\alpha\leq\frac7{40}.
\]
On the reduced strip $x\leq39/100$, we have
$y\geq\sqrt{1-(39/100)^2}>23/25$, and
Lemma~\ref{lem:centerstrip} gives
\[
 E_\alpha(z;a,b)\geq E_\alpha(z;1/2,1/2)
 \geq\min_{\tau\in\mathbb H}E_\alpha(\tau;1/2,1/2),
\]
with equality in the first inequality only at the centered displacement.
Near the triangular lattice, Lemma~\ref{lem:hex-shift-localization}
places a minimizing configuration, after displacement inversion and
integer translations, in the convex prism \eqref{eq:hex-affine-prism}.
Writing
\[
 p=(x,y,a,b),\qquad
 p_0=\left(\frac12,\frac{\sqrt3}{2},\frac13,\frac13\right),
\]
we have $\nabla E_\alpha(p_0)=0$, and
Lemma~\ref{lem:hex-convexity} gives $D^2E_\alpha\succ0$ on the prism
for $4/25\leq\alpha\leq7/40$. Thus, if $p\ne p_0$,
\[
 E_\alpha(p)-E_\alpha(p_0)
 =\int_0^1(1-t)(p-p_0)^{\mathsf T}D^2E_\alpha(p_0+t(p-p_0))(p-p_0)\,dt>0,
\]
contradicting minimality, including at the shape boundary.
On the two regions in Proposition~\ref{prop:bridge},
\[
 E_\alpha(z;a,b)>
 E_\alpha\left(\frac{3+i\sqrt{55}}8;\frac12,\frac12\right)
 \geq\min_{\tau\in\mathbb H}E_\alpha(\tau;1/2,1/2).
\]
These comparisons give Lemma~\ref{lem:first-transition} and its
equality cases.

The function
\[
 \alpha\longmapsto
 \min_{\tau\in\mathbb H}E_\alpha(\tau;1/2,1/2)
       -E_\alpha(e^{i\pi/3};1/3,1/3)
\]
is concave on $[0,1]$ and vanishes at $\alpha=0$.
Together with \eqref{eq:crossing-endpoint-signs}, concavity gives
\eqref{eq:crossing-signs} and the unique crossing
$\alpha_a\in(\frac{1723}{10000},7/40)$. Energy equality and stationarity give
\eqref{eq:crossing-system}. Thus
Lemma~\ref{lem:first-transition} at $\frac{1723}{10000}$ yields
\[
 E_{1723/10000}(z;a,b)\geq E_{1723/10000}(e^{i\pi/3};1/3,1/3),
\]
with equality precisely at the triangular deep holes, modulo the
joint symmetries. For $0<\alpha<\frac{1723}{10000}$, affine dependence on
$\alpha$ gives
\[
\begin{aligned}
 &E_\alpha(z;a,b)-E_\alpha(e^{i\pi/3};1/3,1/3)\\
 &\quad=\left(1-\frac{10000\alpha}{1723}\right)
       \bigl[\theta(1;z)-\theta(1;e^{i\pi/3})\bigr]\\
 &\qquad+\frac{10000\alpha}{1723}
       \bigl[E_{1723/10000}(z;a,b)-E_{1723/10000}(e^{i\pi/3};1/3,1/3)\bigr]
 \geq0.
\end{aligned}
\]
Both brackets are nonnegative by Theorem A and the comparison at
$\frac{1723}{10000}$; equality forces the triangular deep holes.
At $\alpha=0$, the energy is independent of displacement, and
Montgomery's theorem determines the lattice shape.

For the remaining couplings, Lemma~\ref{lem:first-transition},
Theorem~\ref{thm:centered}, and
Theorem~\ref{thm:large-alpha-unrestricted} cover, respectively,
\[
 [\frac{1723}{10000},7/40]\cup[7/40,8/15]\cup[8/15,1]=[\frac{1723}{10000},1].
\]
Together with \eqref{eq:crossing-signs}, they give
\[
 \min_{z\in\mathbb H,\,(a,b)\in\mathbb R^2}E_\alpha(z;a,b)
 =\min_{\tau\in\mathbb H}E_\alpha(\tau;1/2,1/2)
 \qquad(\alpha_a\leq\alpha\leq1).
\]
Finally, the centered identity \eqref{eq:centered-identity}
\[
 E_\alpha(z;1/2,1/2)
 =(1-\alpha)\theta(1;z)
  +2\alpha\theta\left(2;\frac{z+1}{2}\right)
\]
reduces the centered problem to the classification of our previous work~\cite[Theorem~2.1]{LW}, which gives
the rhombic, square, and rectangular phases and their transition
thresholds.

Section~\ref{sec:preliminaries} establishes the joint modular
symmetries, compactness, and common Fourier estimates.
Section~\ref{sec:centered} gives the centered strip and the bridge
comparisons. Section~\ref{sec:triangular} proves the triangular
comparison by displacement localization and integration of the joint
Hessian. Section~\ref{sec:rectangles} treats the rectangle comparison,
locates the first transition, and completes the proof of
Theorem~\ref{thm:classification}. Appendix~\ref{app:certificates}
contains the finite-sum estimates and their error bounds.

\section{Preliminaries: modular invariance and reduction}\label{sec:preliminaries}

In this section, we deduce the symmetries of the functional $E_\alpha(z;a,b)$ with respect to $z$ and $(a,b)$, and then introduce the reduced fundamental domain.

We recall the fundamental-domain definition used by Apostol~\cite{Apo1976} and Evans~\cite{Eva1973}.
\begin{definition} [Evans~{\cite[p.~108]{Eva1973}}]\label{Def}
The fundamental domain associated to group $G$ is a connected domain $\mathcal{D}$ that satisfies:
\begin{itemize}
  \item [(1)] For any $z\in\mathbb{H}$, there exists an element $\pi\in G$ such that $\pi(z)\in\overline{\mathcal{D}}$;
  \item [(2)] Suppose $z_1,z_2\in\mathcal{D}$ and $\pi(z_1)=z_2$ for some $\pi\in G$, then $z_1=z_2$ and $\pi=\pm Id$.
\end{itemize}
\end{definition}

The group $\mathcal{G}$ is defined by
\begin{equation}\aligned\label{GroupG1}
\mathcal{G}: \hbox{the group generated by} \;\;\tau\mapsto -\frac{1}{\tau},\;\; \tau\mapsto \tau+1,\;\;\tau\mapsto -\overline{\tau}.
\endaligned\end{equation}

A fundamental domain for $\mathcal G$ is
\begin{equation}\aligned\label{Fd3}
\mathcal{D}_{\mathcal{G}}:=\{
z\in\mathbb{H}: |z|>1,\; 0<x<\frac{1}{2}
\}.
\endaligned\end{equation}

We then introduce the modular properties of $E_\alpha(z;a,b)$.
\begin{proposition}[Modular invariance and reduction]
\label{Prop21}
Let $\alpha\in\R$, $z\in\mathbb H$, and
$(a,b)\in\mathbb R^2$. Then
\begin{equation}\label{Eq21}
\begin{aligned}
 &E_\alpha(z+1;a-b,b)=E_\alpha(z;a,b),\\
 &E_\alpha(-1/z;-b,a)=E_\alpha(z;a,b),\\
 &E_\alpha(-\bar z;a,-b)=E_\alpha(z;a,b).
\end{aligned}
\end{equation}
and
\begin{equation}\label{Eq22}\begin{aligned}
 &E_\alpha(z;a+1,b)=E_\alpha(z;a,b),\\
  &E_\alpha(z;a,b+1)=E_\alpha(z;a,b),\\
  &E_\alpha(z;1-a,1-b)=E_\alpha(z;a,b)
\end{aligned}\end{equation}

Whenever a minimum is attained,
\begin{equation*}\begin{aligned}
\min_{z\in\mathbb{H}, (a,b)\in\mathbb{R}^2} E_\alpha(z;a,b)
=\min_{z\in\overline{\mathcal{D}_{\mathcal{G}}}, (a,b)\in\mathbb{R}^2} E_\alpha(z;a,b)
=\min_{z\in\overline{\mathcal{D}_{\mathcal{G}}}, (a,b)\in[0,1]^2} E_\alpha(z;a,b).
\end{aligned}\end{equation*}

\end{proposition}

\begin{proof}
Write $z=x+iy$. The series
\[
 E_\alpha(z;a,b)=\sum_{m,n\in\mathbb Z}
 e^{-\pi |mz-n|^2/y}
 \bigl[1+\alpha\cos\bigl(2\pi(ma+nb)\bigr)\bigr]
\]
converges absolutely, since $|1+\alpha\cos t|\leq1+|\alpha|$.
Using
\[
 \Im(-1/z)=\frac{y}{|z|^2},\qquad
 \frac{|m(-1/z)-n|^2}{\Im(-1/z)}=\frac{|nz+m|^2}{y},\qquad
 |m(-\bar z)-n|^2=|mz+n|^2,
\]
we obtain
\[
\begin{aligned}
 E_\alpha(z+1;a-b,b)
 &=\sum_{m,n\in\mathbb Z}e^{-\pi |mz-(n-m)|^2/y}
   \bigl[1+\alpha\cos\bigl(2\pi(ma+(n-m)b)\bigr)\bigr]\\
 &\overset{k=n-m}{=}E_\alpha(z;a,b),\\[4pt]
 E_\alpha(-1/z;-b,a)
 &=\sum_{m,n\in\mathbb Z}e^{-\pi |nz+m|^2/y}
   \bigl[1+\alpha\cos\bigl(2\pi(na-mb)\bigr)\bigr]\\
 &\overset{(p,q)=(n,-m)}{=}E_\alpha(z;a,b),\\[4pt]
 E_\alpha(-\bar z;a,-b)
 &=\sum_{m,n\in\mathbb Z}e^{-\pi |mz+n|^2/y}
   \bigl[1+\alpha\cos\bigl(2\pi(ma-nb)\bigr)\bigr]\\
 &\overset{k=-n}{=}E_\alpha(z;a,b).
\end{aligned}
\]
The three changes of indices are bijections of $\mathbb Z^2$;
this proves \eqref{Eq21}. They also give the corresponding
invariances of $\theta(s;z)$ for every $s>0$.

For $p,q\in\mathbb Z$, we have
\[
\begin{aligned}
 \cos\bigl(2\pi(m(a+p)+n(b+q))\bigr)
 &=\cos\bigl(2\pi(ma+nb)\bigr),\\
 \cos\bigl(2\pi(m(1-a)+n(1-b))\bigr)
 &=\cos\bigl(2\pi(m+n)-2\pi(ma+nb)\bigr)\\
 &=\cos\bigl(2\pi(ma+nb)\bigr).
\end{aligned}
\]
These identities prove \eqref{Eq22} term by term.

The fundamental-domain property of \eqref{Fd3} and
\eqref{Eq21}--\eqref{Eq22} imply
\[
\begin{gathered}
 \forall\,(z;a,b)\in\mathbb H\times\mathbb R^2,\quad
 \exists\,(z';a',b')\in
 \overline{\mathcal D_{\mathcal G}}\times[0,1]^2:\\
 E_\alpha(z';a',b')=E_\alpha(z;a,b).
\end{gathered}
\]
Hence the three infima agree for every real $\alpha$, and the
minima agree whenever attained. For $0\leq\alpha\leq1$,
attainment follows from Lemma~\ref{lem:cusp}.
\end{proof}

We use the Jacobi theta notation of Luo and Wei~\cite[(1.13) and (4.4)]{LW}:
\[
\begin{aligned}
 \vartheta_2(y)&=\sum_{n\in\mathbb Z}e^{-\pi y(n-1/2)^2},\qquad
 \vartheta_3(y)=\sum_{n\in\mathbb Z}e^{-\pi yn^2},\\
 \vartheta_4(y)&=\sum_{n\in\mathbb Z}(-1)^ne^{-\pi yn^2},\qquad
 \vartheta(y;t)=\sum_{n\in\mathbb Z}e^{-\pi yn^2+2\pi int},
\end{aligned}
\]
where $y>0$ and $t\in\mathbb R$.

We next record the analytic properties needed for differentiation
and for comparisons at nonpositive coupling.

\begin{lemma}[Poisson summation]
\label{lem:analytic-properties}

\begin{equation}\label{eq:positive-translated}
 J(z;a,b)=\sum_{p,q\in\mathbb Z}
 \exp\!\left[-\pi\left(
 \frac{(p+a+x(q+b))^2}{y}+y(q+b)^2\right)\right],
\end{equation}
and
\begin{equation}\label{eq:translated-bounds}
 0<J(z;a,b)\leq\theta(1;z).
\end{equation}
Equality in the upper bound holds if and only if
$(a,b)\in\mathbb Z^2$.
\end{lemma}

The centered problem has a useful exact expression in terms of
unshifted theta functions, as in our previous work~\cite{LW}.
\begin{lemma}[The centered identity]\label{lem:centered-identity}
For $z\in\mathbb H$ and $\alpha\in\R$,
\begin{equation}\label{eq:centered-identity}
 E_\alpha(z;1/2,1/2)
 =(1-\alpha)\theta(1;z)+2\alpha\theta\big(2;\frac{z+1}{2}\big).
\end{equation}
\end{lemma}

For nonnegative coupling, the following comparison excludes the
cusp and makes the unrestricted minimization problem compact.

\begin{lemma}[Compactness of the minimization problem]\label{lem:cusp}
For every $0\leq\alpha\leq1$, the unrestricted minimum of $E_\alpha$ is
attained. Every minimizing configuration has a representative satisfying
\[
 0\leq x\leq\frac12,\qquad x^2+y^2\geq1,\qquad
 \frac{\sqrt3}{2}\leq y<2.
\]
\end{lemma}

\begin{proof}
By Proposition~\ref{Prop21}, it suffices to consider
$z\in\overline{\mathcal D_{\mathcal G}}$ and $(a,b)\in[0,1]^2$.
For $t>0$ and $u\in\mathbb R$, the Jacobi product gives
\[
\begin{aligned}
 \vartheta(t;u)
 &=\prod_{k\geq1}(1-e^{-2\pi kt})
 \bigl[(1-e^{-\pi(2k-1)t})^2
       +4e^{-\pi(2k-1)t}\cos^2(\pi u)\bigr]\\
 &\geq\vartheta_4(t)>0.
\end{aligned}
\]
Poisson summation and \eqref{eq:positive-translated} yield
\[
\begin{aligned}
 \theta(1;x+iy)
 &=\sqrt y\sum_{m\in\mathbb Z}e^{-\pi ym^2}\vartheta(y;mx)\\
 &\geq\sqrt y\,[\vartheta_3(y)+\vartheta_4(y)(\vartheta_3(y)-1)],\\
 J(x+iy;a,b)
 &=\sqrt y\sum_{n\in\mathbb Z}e^{-\pi y(n+b)^2}
       \vartheta\bigl(y;a+x(n+b)\bigr)\\
 &\geq\sqrt y\,\vartheta_4(y)
       \sum_{n\in\mathbb Z}e^{-\pi y(n+b)^2}.
\end{aligned}
\]
Moreover,
\[
 \sum_{n\in\mathbb Z}e^{-\pi y(n+b)^2}
 =\frac1{\sqrt y}\vartheta(1/y;b)
 \geq\frac1{\sqrt y}\vartheta_4(1/y)=\vartheta_2(y).
\]
For $y\geq2$, the bounds
\[
 \vartheta_3(y)\geq1+2e^{-\pi y},\qquad
 \vartheta_4(y)\geq1-2e^{-\pi y}>0,\qquad
 \vartheta_2(y)\geq2e^{-\pi y/4}
\]
give
\[
 E_\alpha(x+iy;a,b)\geq
 \sqrt y\,[1+4e^{-\pi y}(1-e^{-\pi y})
               +2\alpha e^{-\pi y/4}(1-2e^{-\pi y})].
\]
Since
\[
\begin{aligned}
 \frac d{dy}\bigl[(2\pi y-1)e^{-\pi y}\bigr]
 &=\pi e^{-\pi y}(3-2\pi y)<0,\\
 \frac d{dy}\bigl[(\pi y-2)e^{-\pi y/4}\bigr]
 &=\frac\pi4e^{-\pi y/4}(6-\pi y)<0,
\end{aligned}
\]
we obtain, for $0\leq\alpha\leq1$,
\[
\begin{aligned}
 &2\sqrt y\,\frac d{dy}
 \left\{\sqrt y\,[1+4e^{-\pi y}(1-e^{-\pi y})
                 +2\alpha e^{-\pi y/4}(1-2e^{-\pi y})]\right\}\\
 &\quad=1-4(2\pi y-1)e^{-\pi y}+4(4\pi y-1)e^{-2\pi y}\\
 &\qquad\quad-\alpha(\pi y-2)e^{-\pi y/4}
               +2\alpha(5\pi y-2)e^{-5\pi y/4}\\
 &\quad\geq1-4(2\pi y-1)e^{-\pi y}-(\pi y-2)e^{-\pi y/4}\\
 &\quad\geq1-4(4\pi-1)e^{-2\pi}-(2\pi-2)e^{-\pi/2}>\frac1{50}.
\end{aligned}
\]

Thus the displayed lower bound is strictly increasing in $y$.
For $t>0$, the geometric and alternating series give
\[
 \vartheta_3(t)\leq1+\frac{2e^{-\pi t}}{1-e^{-3\pi t}},\qquad
 \vartheta_4(t)\leq1-2e^{-\pi t}+2e^{-4\pi t},\qquad
 \vartheta_2(t)\leq\frac{2e^{-\pi t/4}}{1-e^{-2\pi t}}.
\]
At the centered rectangular competitor, these imply
\[
\begin{aligned}
 \theta(1;i\sqrt3)
 &=3^{1/4}\vartheta_3(\sqrt3)^2
 <\frac{71}{50}<\sqrt2\,[1+4e^{-2\pi}(1-e^{-2\pi})],\\
 E_1(i\sqrt3;1/2,1/2)
 &=3^{1/4}\bigl[\vartheta_3(\sqrt3)^2
       +\vartheta_4(\sqrt3)\vartheta_2(\sqrt3)\bigr]\\
 &<\frac{201}{100}
 <\sqrt2\,[1+4e^{-2\pi}(1-e^{-2\pi})
                 +2e^{-\pi/2}(1-2e^{-2\pi})].
\end{aligned}
\]
Affine dependence on $\alpha$ therefore gives, for $y\geq2$,
\[
\begin{aligned}
 E_\alpha(x+iy;a,b)
 &\geq\sqrt2\,[1+4e^{-2\pi}(1-e^{-2\pi})
                   +2\alpha e^{-\pi/2}(1-2e^{-2\pi})]\\
 &>(1-\alpha)\theta(1;i\sqrt3)
       +\alpha E_1(i\sqrt3;1/2,1/2)\\
 &=E_\alpha(i\sqrt3;1/2,1/2).
\end{aligned}
\]
The Gaussian series converge uniformly on
\[
 \{z\in\overline{\mathcal D_{\mathcal G}}:\Im z\leq2\}
 \times[0,1]^2.
\]
Continuity gives a minimum, and the strict comparison excludes $y\geq2$.
\end{proof}

The next lemma is the common Fourier lower bound used to compare
different lattice shapes. Minimizing the contribution of four Fourier
pairs eliminates the displacement, while positivity controls the
omitted terms.

\begin{lemma}[Elimination of the displacement]\label{lem:four}
Let $0\leq\alpha\leq1$ and $z=x+iy\in\mathbb H$.
For $(m,n)\in\mathbb Z^2$, define
\begin{equation}\label{eq:lattice-quadratic-form}
 Q_{m,n}=Q_{m,n}(x,y):=\frac{|mz-n|^2}{y}
                       =m^2y+\frac{(mx-n)^2}{y},
\end{equation}
and put
\[
 A=e^{-\pi(y+x^2/y)},\quad B=e^{-\pi/y},\quad
 C=e^{-\pi(y+(x-1)^2/y)},\quad D=e^{-\pi(y+(x+1)^2/y)}.
\]
Assume that
\[
 x\geq0,\qquad A^2>4CD,\qquad B^2>4CD.
\]
Then, for every $(a,b)\in\mathbb R^2$,
\begin{align}\label{eq:lower}
 E_\alpha(z;a,b)\geq{}&1+\alpha+2(A+B+C+D)\\
 &+2\alpha\min_{-1\leq t\leq1}
 \left[Bt-\sqrt{A^2+(C-D)^2+2A(C+D)t+4CDt^2}\right]\nonumber\\
 &+(1-\alpha)
 \sum_{\substack{|m|,|n|\leq2\\
 (m,n)\notin\{(0,0),\pm(1,0),\pm(0,1),\pm(1,1),\pm(1,-1)\}}}
 e^{-\pi Q_{m,n}}.\nonumber
\end{align}
The right-hand side is locally Lipschitz in $(x,y)$, and the scalar
minimum is attained uniquely at
\[
\begin{aligned}
 t&=\max\{-1,\min\{1,t_0\}\},\\
 t_0&=
 \frac{A^2(C+D)^2-B^2[A^2+(C-D)^2]}
 {\sqrt{B^2-4CD}
 [B(C-D)\sqrt{A^2-4CD}+A(C+D)\sqrt{B^2-4CD}]}.
\end{aligned}
\]
\end{lemma}

\begin{proof}
With $u=2\pi a$ and $v=2\pi b$, the four retained pairs contribute
\[
 A\cos u+B\cos v+C\cos(u+v)+D\cos(u-v).
\]
For $t=\cos v$, minimization in $u$ gives precisely the minimum
in \eqref{eq:lower}. Every omitted coefficient of $E_\alpha$ is
$1+\alpha\cos(2\pi(ma+nb))\geq1-\alpha\geq0$.
Retaining the stated finite remainder proves the inequality.

Let $f(t)$ denote the function inside the minimum. The assumptions give
\[
 C/D=e^{4\pi x/y}\geq1,\qquad
 A(C+D)\geq2A\sqrt{CD}>4CD.
\]
If $x=0$, then $C=D$, $A>2C$, and $B>2C$. Thus
\[
 f(t)=(B-2C)t-A,\qquad t_0=-\frac{A}{2C}<-1,
\]
and the unique minimum is at $t=-1$, as asserted.
Assume henceforth that $C>D$. The radicand is positive for
$-1<t<1$, and
\[
 f''(t)=
 \frac{(C-D)^2(A^2-4CD)}
 {[A^2+(C-D)^2+2A(C+D)t+4CDt^2]^{3/2}}>0
 \qquad(-1<t<1).
\]
Thus $f$ is strictly convex on $[-1,1]$.
Solving $f'(t)=0$ on the branch containing $[-1,1]$ gives $t=t_0$.
Strict convexity then gives the unique constrained minimizer
$t=\max\{-1,\min\{1,t_0\}\}$.
The interior stationary value is
\[
 -\frac{A^2B^2+(C-D)^2(A^2+B^2-4CD)}
 {AB(C+D)+(C-D)\sqrt{(A^2-4CD)(B^2-4CD)}}.
\]
If $f'(-1)\geq0$, the minimum equals
$-B-|A-C-D|$. When the unconstrained stationary point lies outside
$[-1,1]$, its value is still a valid lower bound for the constrained
minimum; the projected point gives the exact value.
These formulas avoid cancellation as $D$ tends to zero.

At a minimizing $t$ for which
$A^2+(C-D)^2+2A(C+D)t+4CDt^2>0$, the corresponding cosines are
\[
\begin{aligned}
 \cos u&=-\frac{A+(C+D)t}
 {\sqrt{A^2+(C-D)^2+2A(C+D)t+4CDt^2}},&\cos v&=t,\\
 \cos(u+v)&=\frac{-[A+(C+D)t]t-(C-D)(1-t^2)}
 {\sqrt{A^2+(C-D)^2+2A(C+D)t+4CDt^2}},\\
 \cos(u-v)&=\frac{-[A+(C+D)t]t+(C-D)(1-t^2)}
 {\sqrt{A^2+(C-D)^2+2A(C+D)t+4CDt^2}}.
\end{aligned}
\]
The shape derivative of the minimized contribution is obtained by
multiplying the derivatives of $A,B,C,D$ by
$2(1+\alpha\cos u)$ and the three analogous factors.
The finite remainder is differentiated termwise.

To justify the derivative bounds at changes of branch, the minimum is
taken over a compact displacement torus of smooth functions whose shape
derivatives are bounded on compact shape rectangles. It is therefore
locally Lipschitz. Along any line segment it is absolutely continuous,
and its derivative almost everywhere is bounded by the directional
derivatives at the minimizing displacements. The displayed cosine
formulas enclose those derivatives; when the radicand interval meets
zero, using $[-1,1]$ for every cosine remains valid.
Thus integrating the resulting derivative bounds along segments gives
rigorous shape-box estimates, including boxes meeting a branch change.
\end{proof}

The inequality \eqref{eq:lower} itself holds for every $z\in\mathbb H$;
the coefficient assumptions are used for the explicit minimization.
In all subsequent applications, Lemma~\ref{lem:cusp} allows us to work
on $z\in\overline{\mathcal D_{\mathcal G}}$ with $y\leq2$. On this region,
$x\geq0$ and $y+x^2/y\geq\max\{y,1/y\}\geq1$, so
\[
 \frac{A^2}{4CD}=\frac14e^{2\pi/y}\geq\frac14e^\pi>1,
 \qquad
 \frac{B^2}{4CD}=\frac14e^{2\pi(y+x^2/y)}
 \geq\frac14e^{2\pi}>1.
\]

\section{Comparisons with $E_\alpha(z;1/2,1/2)$}\label{sec:centered}

In this section, we prove that every minimizer is centered for
$7/40\leq\alpha\leq23/25$. We also obtain the comparisons needed
to connect this range with the first transition.

\begin{theorem}\label{thm:centered}
For $7/40\leq\alpha\leq23/25$,
\[
 \min_{z\in\mathbb H,\ (a,b)\in\mathbb R^2/\mathbb Z^2}
 E_\alpha(z;a,b)
 =\min_{z\in\mathbb H}E_\alpha(z;1/2,1/2).
\]
Every minimizing configuration is equivalent to a centered minimizing
configuration. In particular, the square lattice is the unique minimizing
shape for $2/5\leq\alpha\leq23/25$, modulo the stated symmetries.
\end{theorem}

By Lemma~\ref{lem:cusp}, it suffices to consider reduced shapes
with $y<2$. A quadratic Fourier estimate forces the centered
displacement on a strip; outside that strip,
Lemma~\ref{lem:four} gives strict comparisons with fixed centered
configurations. We first prove the displacement estimate.

\begin{lemma}\label{lem:centerstrip}
For
\[
 0\leq x\leq\frac{39}{100},\qquad
 \frac{23}{25}\leq y\leq2,
\]
the unique minimum of $J(x+iy;\cdot,\cdot)$ is attained at
$(1/2,1/2)$ modulo integer translations.
\end{lemma}

\begin{proof}
With $Q_{m,n}$ as in \eqref{eq:lattice-quadratic-form}, put
$u=2\pi(a-1/2)$ and $v=2\pi(b-1/2)$. Since
\[
 \cos\bigl(2\pi(ma+nb)\bigr)=(-1)^{m+n}\cos(mu+nv),
\]
we have the exact decomposition
\[
\begin{aligned}
 J(z;a,b)-J(z;1/2,1/2)
 ={}&\sum_{\substack{m,n\in\mathbb Z\\m+n\ \mathrm{odd}}}
 e^{-\pi Q_{m,n}}\bigl[1-\cos(mu+nv)\bigr]\\
 &-\sum_{\substack{m,n\in\mathbb Z\\m+n\ \mathrm{even}}}
 e^{-\pi Q_{m,n}}\bigl[1-\cos(mu+nv)\bigr].
\end{aligned}
\]
For integers $m,n$,
\[
 |e^{i(mu+nv)}-1|
 \leq|e^{imu}-1|+|e^{inv}-1|
 \leq |m|\,|e^{iu}-1|+|n|\,|e^{iv}-1|.
\]
Using $|e^{it}-1|^2=2(1-\cos t)$ gives
\[
 1-\cos(mu+nv)\leq
 \bigl(|m|\sqrt{1-\cos u}+|n|\sqrt{1-\cos v}\bigr)^2.
\]
Retaining $(m,n)=\pm(1,0),\pm(0,1)$ in the nonnegative odd sum
and applying this estimate to the even sum yields
\begin{align*}
 J(z;a,b)-J(z;1/2,1/2)\geq{}&
 \left(2e^{-\pi Q_{1,0}}-
 \sum_{\substack{m,n\in\mathbb Z\\m+n\ \mathrm{even}}}
 m^2e^{-\pi Q_{m,n}}\right)(1-\cos u)\\
 &+\left(2e^{-\pi Q_{0,1}}-
 \sum_{\substack{m,n\in\mathbb Z\\m+n\ \mathrm{even}}}
 n^2e^{-\pi Q_{m,n}}\right)(1-\cos v)\\
 &-2\left(\sum_{\substack{m,n\in\mathbb Z\\m+n\ \mathrm{even}}}
 |mn|e^{-\pi Q_{m,n}}\right)
 \sqrt{(1-\cos u)(1-\cos v)}.
\end{align*}
By Lemma~\ref{lem:cert-strip}, this quadratic form in
$\sqrt{1-\cos u}$ and $\sqrt{1-\cos v}$ is positive definite.
Hence the difference is strictly positive unless
\[
 \cos u=\cos v=1
 \quad\Longleftrightarrow\quad
 (a,b)\equiv(1/2,1/2)\pmod{\mathbb Z^2}.
\]
Periodicity gives equality for these displacements.
\end{proof}

\begin{proof}[Proof of Theorem~\ref{thm:centered}]
By Lemma~\ref{lem:cusp}, it suffices to consider reduced shapes with
$y<2$. For $7/40\leq\alpha\leq2/5$, the six comparisons in
Lemma~\ref{lem:cert-global} exclude $x\geq39/100$, since their
intervals cover $[7/40,2/5]$. On the remaining strip,
$y\geq\sqrt{1-(39/100)^2}>23/25$,
so Lemma~\ref{lem:centerstrip} forces the centered displacement.

For $2/5\leq\alpha\leq23/25$, Lemma~\ref{lem:cert-global}
excludes $x\geq3/10$ by comparison with the centered square.
On the remaining shapes,
$y\geq\sqrt{1-(3/10)^2}>23/25$, so
Lemma~\ref{lem:centerstrip} again applies. Every unrestricted
minimizer is therefore centered. The centered classification of
\cite[Theorem~2.1]{LW} gives the
asserted square phase and its uniqueness.
\end{proof}

The preceding comparisons apply from $\alpha=7/40$ onwards. To
connect them with the first transition, we use one centered rhombic
competitor on the two regions outside a small triangular neighborhood.

\begin{proposition}\label{prop:bridge}
Let $1723/10000\leq\alpha\leq7/40$. Every reduced shape in either region
\[
 \frac{39}{100}\leq x\leq\frac{19}{40},\quad y<2,
 \qquad\text{or}\qquad
 \frac{19}{40}\leq x\leq\frac12,\quad
 \frac{891}{1000}\leq y<2
\]
satisfies, for every displacement,
\[
 E_\alpha(x+iy;a,b)>
 E_\alpha\left(\frac{3+i\sqrt{55}}8;\frac12,\frac12\right).
\]
\end{proposition}

\begin{proof}
The third comparison in Lemma~\ref{lem:cert-global} gives the
asserted inequality uniformly on both regions.
\end{proof}

\section{Comparison with $E_\alpha(e^{i\pi/3};\frac13,\frac13)$}\label{sec:triangular}

The comparisons in Section~\ref{sec:centered} leave a neighborhood
of the triangular lattice. In this section, we compare every
configuration in this neighborhood with the triangular deep-hole
configuration for $\alpha\in[4/25,7/40]$, covering $\alpha_a$.

\begin{theorem}[Comparison in a triangular neighborhood]
\label{cor:hex-cap-comparison}
If
\[
 \frac{19}{40}\le x\le\frac12,\qquad
 \frac{\sqrt3}{2}\le y\le \frac{\sqrt3}{2}+\frac1{40},\qquad
 \frac4{25}\le\alpha\le\frac7{40},
\]
then, for every $(a,b)\in\mathbb R^2/\mathbb Z^2$,
\[
 E_\alpha(x+iy;a,b)\ge
 E_\alpha\left(\frac12+\frac{i\sqrt3}{2};\frac13,\frac13\right).
\]
Equality holds if and only if $x=1/2$, $y=\frac{\sqrt3}{2}$, and
$(a,b)\equiv\pm(1/3,1/3)\pmod{\mathbb Z^2}$.
\end{theorem}

Three retained Fourier pairs localize the displacement minimizers
near the two deep holes. Positivity of the joint Hessian then gives
the energy comparison and its equality cases. The first step is
the following localization lemma.

\begin{lemma}[Localization of the displacement]
\label{lem:hex-shift-localization}
Let $19/40\le x\le1/2$, and
$\frac{\sqrt3}{2}\le y\le \frac{\sqrt3}{2}+1/40$. Every global minimizer of $J(x+iy;\cdot)$
has, after simultaneous inversion of its displacement if necessary,
a representative satisfying
\[
\begin{aligned}
 a&=\frac13-\left(x-\frac12\right)-\frac{29}{50}\left(y-\frac{\sqrt3}{2}\right)+r_a,\\
 b&=\frac13+\frac{29}{25}\left(y-\frac{\sqrt3}{2}\right)+r_b,
\end{aligned}
\qquad
 |r_a|<\frac3{400},\quad |r_b|<\frac3{500}.
\]
\end{lemma}

\begin{proof}
With $A,B,C$ as in Lemma~\ref{lem:four} and $Q_{m,n}$ as in
\eqref{eq:lattice-quadratic-form}, put
\[
 P=\sqrt{AB/C},\qquad Q=\sqrt{AC/B},\qquad R=\sqrt{BC/A}.
\]
For $U=2\pi a$ and $V=2\pi b$, define
\[
 \mathcal R(a,b)=
 \sum_{\substack{(m,n)\in\mathbb Z^2\\
 (m,n)\notin\{(0,0),\,\pm(1,0),\,\pm(0,1),\,\pm(1,1)\}}}
 e^{-\pi Q_{m,n}}\cos(mU+nV).
\]
The exact decomposition is
\begin{align}
 J(x+iy;a,b)
 &=1+2\{A\cos U+B\cos V+C\cos(U+V)\}+\mathcal R(a,b)\\
 &=1-P^2-Q^2-R^2+|P+Qe^{iU}+Re^{-iV}|^2
       +\mathcal R(a,b).\label{eq:hex-phasor}
\end{align}
The strict triangle inequalities for $P,Q,R$, verified in
Lemma~\ref{lem:cert-localization}, imply that the last square vanishes
at exactly two points on the displacement torus.
Denote them by $\pm(a_0,b_0)$, choosing
$U_0=2\pi a_0$, $V_0=2\pi b_0$ in $(0,\pi)$. They are
determined by the exact identities
\begin{equation}\label{eq:hex-phasor-centers}
\begin{aligned}
 \cos U_0&=\frac{R^2-P^2-Q^2}{2A},&
 \cos V_0&=\frac{Q^2-P^2-R^2}{2B},\\
 \cos(U_0+V_0)&=\frac{P^2-Q^2-R^2}{2C}.
\end{aligned}
\end{equation}
Indeed, the first two equations determine the angles in $(0,\pi)$;
the triangle identities give $Q\sin U_0=R\sin V_0$ and
$P+Q\cos U_0+R\cos V_0=0$.

For fixed $(x,y)$, set
\[
 \operatorname{osc}_{a,b}\mathcal R
 :=\max_{(a,b)\in[0,1]^2}\mathcal R(a,b)
   -\min_{(a,b)\in[0,1]^2}\mathcal R(a,b),
\]
\[
 \nabla_{a,b}\mathcal R=(\partial_a\mathcal R,\partial_b\mathcal R),
 \qquad
 \|\nabla_{a,b}\mathcal R\|_2
 :=\sqrt{|\partial_a\mathcal R|^2+|\partial_b\mathcal R|^2}.
\]
Lemma~\ref{lem:cert-localization} gives
\begin{equation}\label{eq:hex-cap-tail}
 \operatorname{osc}_{a,b}\mathcal R<\frac{27}{100000},\qquad
 \sup_{(a,b)\in[0,1]^2}\|\nabla_{a,b}\mathcal R(a,b)\|_2<\frac1{625},
\end{equation}
together with
\begin{equation}\label{eq:hex-cap-coarse}
\begin{gathered}
 A>\frac1{40},\quad B>\frac1{38},\quad C>\frac1{44},\qquad
 P<\frac{19}{100},\quad Q<\frac16,\quad R<\frac7{40},\\
 -\frac{16}{25}<\cos U_0<0,\qquad
 -\frac23<\cos V_0<0,\qquad \cos(U_0+V_0)<-\frac29.
\end{gathered}
\end{equation}
Lemma~\ref{lem:cert-localization} establishes these bounds uniformly
on the cap. In particular, $U_0,V_0\in(\pi/2,3\pi/4)$.

Let $(a,b)$ minimize $J$ at the fixed shape. Its existence follows
from continuity on the compact torus. Comparison with $(a_0,b_0)$
and \eqref{eq:hex-phasor} gives
\[
 |P+Qe^{iU}+Re^{-iV}|^2
 \le\mathcal R(a_0,b_0)-\mathcal R(a,b)
 <\frac{27}{100000}.
\]
Hence
\[
 |P+Qe^{iU}+Re^{-iV}|<\sqrt{\frac{27}{100000}}<\frac{33}{2000}.
\]
The triangle inequality implies
$\big||P+Qe^{iU}|-R\big|<\frac{33}{2000}$; squaring and using
\eqref{eq:hex-phasor-centers} gives
\[
 |\cos U-\cos U_0|
 <\frac{33(2R+33/2000)}{4000A}.
\]
The other two pairings give, in the same way,
\begin{equation}\label{eq:hex-angular-errors}
 \begin{aligned}
 |\cos V-\cos V_0|
 &<\frac{33(2Q+33/2000)}{4000B},\\
 |\cos(U+V)-\cos(U_0+V_0)|
 &<\frac{33(2P+33/2000)}{4000C}.
 \end{aligned}
\end{equation}
The elementary inequalities
\[
\begin{aligned}
 \frac{\sqrt{369}}{25}\sin\frac{7\pi}{125}
 -\frac{16}{25}\left(1-\cos\frac{7\pi}{125}\right)
 &>\frac{33}{100}\left(\frac7{20}+\frac{33}{2000}\right),\\
 \frac{\sqrt5}{3}\sin\frac{27\pi}{500}
 -\frac23\left(1-\cos\frac{27\pi}{500}\right)
 &>\frac{627}{2000}\left(\frac13+\frac{33}{2000}\right)
\end{aligned}
\]
and monotonicity of cosine on $[0,\pi]$ imply
\[
 |a\mp a_0|<\frac7{250},\qquad |b\mp b_0|<\frac{27}{1000},
\]
with the signs initially chosen independently.

The last inequality in \eqref{eq:hex-angular-errors}, with
\eqref{eq:hex-cap-coarse}, gives
\[
 \cos(U+V)<-\frac29+\frac{363}{1000}
                   \left(\frac{19}{50}+\frac{33}{2000}\right)<0.
\]
If the
signs were opposite, the representative of $U+V$ closest to zero
would satisfy
\[
 |U+V|<\frac\pi4+2\pi\left(\frac7{250}+\frac{27}{1000}\right)
       =\frac{9\pi}{25}<\frac\pi2,
\]
because $U_0,V_0\in(\pi/2,3\pi/4)$. Its cosine would then be
positive, a contradiction. After simultaneous inversion, we may
therefore suppose that
\[
 |a-a_0|<\frac7{250},\qquad |b-b_0|<\frac{27}{1000}.
\]

The displacement Hessian satisfies the uniform estimate
\begin{equation}\label{eq:hex-cap-shift-convexity}
 D^2_{a,b}J(x+iy;a,b)\ge\frac{13}{20}I_2
 \quad\text{when }|a-a_0|,|b-b_0|\le\frac3{100}.
\end{equation}
This is the displacement-Hessian assertion of
Lemma~\ref{lem:cert-hessians}. Set $v=(a-a_0,b-b_0)^{\mathsf T}$.
The segment $(a_0,b_0)+tv$, $0\le t\le1$, lies in the displayed
rectangle. Since $\nabla_{a,b}J(a,b)=0$ and
$\nabla_{a,b}J(a_0,b_0)=\nabla_{a,b}\mathcal R(a_0,b_0)$,
\[
\begin{aligned}
 \frac{13}{20}\|v\|_2^2
 &\le\int_0^1 v^{\mathsf T}
 D^2_{a,b}J\bigl((a_0,b_0)+tv\bigr)v\,dt\\
 &=-v\cdot\nabla_{a,b}\mathcal R(a_0,b_0)
 \le\|v\|_2\|\nabla_{a,b}\mathcal R(a_0,b_0)\|_2
 <\frac{\|v\|_2}{625}\qquad(v\ne0).
\end{aligned}
\]
Consequently,
\begin{equation}\label{eq:hex-cap-shift-distance}
 \|(a,b)-(a_0,b_0)\|_2<\frac4{1625}.
\end{equation}
Lemma~\ref{lem:cert-localization} also gives the affine enclosures
\begin{equation}\label{eq:hex-cap-affine-centers}
\begin{aligned}
 \left|a_0-\left\{\frac13-\left(x-\frac12\right)
                   -\frac{29}{50}\left(y-\frac{\sqrt3}{2}\right)\right\}\right|
 &\le\frac1{200},\\
 \left|b_0-\left\{\frac13+\frac{29}{25}\left(y-\frac{\sqrt3}{2}\right)\right\}\right|
 &\le\frac7{2000}.
\end{aligned}
\end{equation}
Combining
\eqref{eq:hex-cap-shift-distance} and
\eqref{eq:hex-cap-affine-centers} proves
\[
 |r_a|<\frac1{200}+\frac4{1625}<\frac3{400},\qquad
 |r_b|<\frac7{2000}+\frac4{1625}<\frac3{500}.
\]
\end{proof}

\begin{lemma}[Positive definiteness of the joint Hessian]
\label{lem:hex-convexity}
For $4/25\le\alpha\le7/40$, the Hessian of $E_\alpha$ in the
real variables $(x,y,a,b)$ is positive definite on the affine prism
\begin{equation}\label{eq:hex-affine-prism}
\begin{gathered}
 \frac{19}{40}\le x\le\frac12,\qquad \frac{\sqrt3}{2}\le y\le \frac{\sqrt3}{2}+\frac1{40},\\
 a=\frac13-\left(x-\frac12\right)-\frac{29}{50}\left(y-\frac{\sqrt3}{2}\right)+r_a,
 \qquad b=\frac13+\frac{29}{25}\left(y-\frac{\sqrt3}{2}\right)+r_b,\\
 |r_a|\le\frac3{400},\qquad |r_b|\le\frac3{500}.
\end{gathered}
\end{equation}
\end{lemma}

\begin{proof}
For $Q_{m,n}$ from \eqref{eq:lattice-quadratic-form}, the Hessian of
$e^{-\pi Q_{m,n}}[1+\alpha\cos(2\pi(ma+nb))]$ has the following
blocks, corresponding to $(x,y)$ and $(a,b)$:
\begin{equation}\label{eq:hex-joint-hessian}
\begin{aligned}
 H_{11}&=e^{-\pi Q_{m,n}}[1+\alpha\cos(2\pi(ma+nb))]\\
 &\qquad\times[\pi^2\nabla Q_{m,n}(\nabla Q_{m,n})^{\mathsf T}
                   -\pi D^2Q_{m,n}],\\
 H_{12}&=2\alpha\pi^2e^{-\pi Q_{m,n}}
 \sin(2\pi(ma+nb))\nabla Q_{m,n}\begin{pmatrix}m&n\end{pmatrix},\\
 H_{22}&=-4\alpha\pi^2e^{-\pi Q_{m,n}}
 \cos(2\pi(ma+nb))\begin{pmatrix}m^2&mn\\mn&n^2\end{pmatrix},
\end{aligned}
\end{equation}
where the derivatives of $Q_{m,n}$ are taken in $(x,y)$:
\[
 \nabla Q_{m,n}=\left(\frac{2m(mx-n)}y,
                 m^2-\frac{(mx-n)^2}{y^2}\right)^{\mathsf T},
\]
\[
 D^2Q_{m,n}=
 \begin{pmatrix}
 2m^2/y&-2m(mx-n)/y^2\\
 -2m(mx-n)/y^2&2(mx-n)^2/y^3
 \end{pmatrix}.
\]
The series of these Hessians converges uniformly on the prism;
the tail estimate is given in \eqref{eq:hex-joint-hessian-tail}.
Lemma~\ref{lem:cert-hessians} proves the endpoint estimates.
On \eqref{eq:hex-affine-prism}, therefore,
\[
 D^2E_\alpha
 =\frac{7/40-\alpha}{7/40-4/25}D^2E_{4/25}
  +\frac{\alpha-4/25}{7/40-4/25}D^2E_{7/40}\succ0
 \qquad\left(\frac4{25}\le\alpha\le\frac7{40}\right).
\]
\end{proof}

\begin{proof}[Proof of Theorem~\ref{cor:hex-cap-comparison}]
Fix $\alpha\in[4/25,7/40]$, and regard $E_\alpha$ as a function
of the four real variables $(x,y,a,b)$.
It attains its minimum on the closed shape rectangle in the statement
times the displacement torus. Choose such a minimizing configuration
$p=(x,y,a,b)$. Since $\alpha>0$, its displacement minimizes $J$
at the fixed shape. Lemma~\ref{lem:hex-shift-localization} places
$p$, after simultaneous displacement inversion and integer
translations, in the affine prism \eqref{eq:hex-affine-prism}.
This prism is convex and contains
\[
 p_0=\left(\frac12,\frac{\sqrt3}{2},\frac13,\frac13\right).
\]

The triangular configuration is stationary in all four variables.
Indeed, the joint modular symmetry and integer periodicity give
\[
 E_\alpha\circ\Phi=E_\alpha,\qquad
 \Phi(z;a,b)=\left(1-\frac1z;1-a-b,a\right),\qquad
 \Phi(p_0)=p_0.
\]
At $p_0$, the shape derivative is multiplication by
$e^{-2\pi i/3}$, and the displacement derivative is
$\left(\begin{smallmatrix}-1&-1\\1&0\end{smallmatrix}\right)$.
Thus $\det(I-D\Phi(p_0))=9$, and differentiation yields
\[
 (I-D\Phi(p_0)^{\mathsf T})\nabla E_\alpha(p_0)=0,
 \qquad \nabla E_\alpha(p_0)=0.
\]

Suppose that $p\ne p_0$. If the minimum is attained in the
interior of the shape rectangle, then $\nabla E_\alpha(p)=0$.
Lemma~\ref{lem:hex-convexity} and integration give the contradiction
\[
\begin{aligned}
 0&=(p-p_0)\cdot[\nabla E_\alpha(p)-\nabla E_\alpha(p_0)]\\
 &=\int_0^1(p-p_0)^{\mathsf T}
 D^2E_\alpha(p_0+t(p-p_0))(p-p_0)\,dt>0.
\end{aligned}
\]
If $p$ lies on the shape boundary, convexity keeps
$p-s(p-p_0)$ feasible for $0\le s\le1$. Minimality gives
\[
\begin{aligned}
 0&\le\left.\frac d{ds}E_\alpha(p-s(p-p_0))\right|_{s=0^+}\\
 &=-\int_0^1(p-p_0)^{\mathsf T}
 D^2E_\alpha(p_0+t(p-p_0))(p-p_0)\,dt<0,
\end{aligned}
\]
again a contradiction. Every minimum is therefore the triangular
deep-hole configuration, up to inversion and integer translations.
The two deep holes have the same energy by inversion, proving the
comparison and all equality cases.
\end{proof}

\section{Comparison with $E_\alpha(iy;1/2,1/2)$}\label{sec:rectangles}

In this section, we reduce the minimization problem to centered
rectangles for $8/15\leq\alpha\leq1$. This range overlaps that of
Section~\ref{sec:centered}. We then characterize the first transition
and complete the proof of Theorem~\ref{thm:classification}.

\begin{theorem}\label{thm:large-alpha-unrestricted}
Let $8/15\leq\alpha\leq1$. After a change of lattice basis and reflection,
every global minimizer of $E_\alpha(z;a,b)$ has
\[
 z=iy,\quad y\geq1,\qquad (a,b)\equiv(1/2,1/2)\pmod{\mathbb Z^2}.
\]
The minimizing lattice is square for
$8/15\leq\alpha\leq\alpha_c$ and centered rectangular for
$\alpha_c<\alpha\leq1$, with the aspect ratio $y_\alpha$ defined in Section~\ref{sec:introduction}. At
$\alpha=1$ the aspect ratio is $\sqrt3$.
\end{theorem}

The main estimate compares a configuration with the centered
rectangle at the same height. One-dimensional theta estimates show
that the translated interaction controls the loss caused by lattice
shear, uniformly over all displacements.

\begin{lemma}\label{lem:same-height}
Let $|x|\leq1/2$, $y\geq\frac{\sqrt3}{2}$, and $\alpha\geq8/15$. Then
\[
 E_\alpha(x+iy;a,b)\geq E_\alpha(iy;1/2,1/2).
\]
Equality holds exactly when $x=0$ and
$(a,b)\equiv(1/2,1/2)\pmod{\mathbb Z^2}$.
\end{lemma}

\begin{proof}
Since $y\geq\sqrt3/2$, we have $e^{-\pi y}<1/15$. The
alternating series and the estimates
$n^2\geq4+5(n-2)$, $n^2\leq4\,3^{n-2}$ for $n\geq2$ give
\begin{equation}\label{eq:large-alpha-tails}
 \vartheta_4(y)\geq1-2e^{-\pi y},\qquad
 \sum_{n\geq1}n^2e^{-\pi yn^2}
 \leq e^{-\pi y}(1+5e^{-3\pi y}).
\end{equation}
Indeed,
\[
 \sum_{n\geq2}n^2e^{-\pi yn^2}
 \leq\frac{4e^{-4\pi y}}{1-3e^{-5\pi y}}<5e^{-4\pi y}.
\]
The Jacobi product implies that $\vartheta(y;\cdot)$ has its
unique minimum $\vartheta_4(y)$ at $1/2$ modulo $\mathbb Z$.
More precisely,
\begin{equation}\label{eq:quantified-one-d}
 \vartheta(y;t)\geq\vartheta_4(y)
 +4e^{-\pi y}(1-5e^{-3\pi y})\cos^2(\pi t).
\end{equation}
The $k=1$ pair contributes $4e^{-\pi y}\cos^2(\pi t)$ to the
difference. Every odd $k$ contributes a nonnegative quantity,
whereas for even $k$, $|\sin(ku)|\leq k|\sin u|$ bounds its
contribution below by
$-4k^2e^{-\pi yk^2}\cos^2(\pi t)$.
Thus \eqref{eq:large-alpha-tails} proves
\eqref{eq:quantified-one-d}.

Rowwise Poisson summation gives
\[
 \theta(1;x+iy)=\sqrt y\sum_{m,k\in\mathbb Z}
 e^{-\pi y(m^2+k^2)}\cos(2\pi mkx).
\]
Hence
\begin{equation}\label{eq:theta-shear-loss}
\begin{aligned}
 0\leq\theta(1;iy)-\theta(1;x+iy)
 &=8\sqrt y\sum_{m,k\geq1}
 e^{-\pi y(m^2+k^2)}\sin^2(\pi mkx)\\
 &\leq8\sqrt y\,e^{-2\pi y}(1+5e^{-3\pi y})^2
 \sin^2(\pi x).
\end{aligned}
\end{equation}
The translated representation of $J$, followed by horizontal
Poisson summation, yields
\[
 J(x+iy;a,b)=\sqrt y\sum_{n\in\mathbb Z}
 e^{-\pi y(n+b)^2}\vartheta\bigl(y;a+x(n+b)\bigr),
 \quad
 J(iy;1/2,1/2)=\sqrt y\,\vartheta_4(y)\vartheta_2(y).
\]
By periodicity, take $0\leq b\leq1$. The Jacobi product after
one-dimensional Poisson summation shows that
$\sum_ne^{-\pi y(n+b)^2}$ decreases strictly on $[0,1/2]$,
is symmetric about $1/2$, and has minimum $\vartheta_2(y)$.
Retaining the $n=0,-1$ corrections in
\eqref{eq:quantified-one-d} gives
\begin{equation}\label{eq:two-row-penalty}
\begin{aligned}
 &J(x+iy;a,b)-\sqrt y\,\vartheta_4(y)\vartheta_2(y)\\
 &\quad\geq\sqrt y\left[
 \vartheta_4(y)\left(\sum_{n\in\mathbb Z}e^{-\pi y(n+b)^2}
 -\vartheta_2(y)\right)
 +\frac{4e^{-\pi y}(1-5e^{-3\pi y})}
 {e^{\pi yb^2}+e^{\pi y(1-b)^2}}\sin^2(\pi x)\right].
\end{aligned}
\end{equation}
Here, for $\lambda,\mu>0$,
\[
 \min_u\{\lambda\cos^2u+\mu\cos^2(u-\pi x)\}
 =\frac{\lambda+\mu-
 \sqrt{(\lambda+\mu)^2-4\lambda\mu\sin^2(\pi x)}}2
 \geq\frac{\lambda\mu}{\lambda+\mu}\sin^2(\pi x).
\]

Suppose first that $\sqrt3/2\leq y\leq1$. Since
$e^{-\pi/y}<1/23$, Poisson summation and
\eqref{eq:quantified-one-d}, applied at $1/y$, imply
\[
 \sum_ne^{-\pi y(n+b)^2}-\vartheta_2(y)
 \geq\frac{4e^{-\pi/y}(1-5e^{-3\pi/y})}{\sqrt y}
 \sin^2\bigl(\pi(b-1/2)\bigr).
\]
Moreover,
\[
 \frac1{e^{\pi yb^2}+e^{\pi y(1-b)^2}}
 =\frac{e^{-\pi y/4}}2
 \frac{e^{-\pi y(b-1/2)^2}}{\cosh\bigl(\pi y(b-1/2)\bigr)}
 \geq\frac{e^{-\pi y/4}}2
 \cos^2\bigl(\pi(b-1/2)\bigr).
\]
Indeed, $\log\cosh u\leq u^2/2$, $\log\cos u\leq-u^2/2$
for $|u|<\pi/2$, and
$\pi y+\pi^2y^2/2\leq\pi+\pi^2/2<\pi^2$; the endpoints
follow by continuity. Also,
\[
 \frac{4\vartheta_4(y)e^{-\pi/y}(1-5e^{-3\pi/y})}{\sqrt y}
 >2e^{-5\pi y/4}(1-5e^{-3\pi y}),
\]
because
\[
 e^{-\pi(5y/4-1/y)}
 \leq e^{\pi/(8\sqrt3)}<\frac43
 <\frac53<2\frac{13}{15}\left(1-\frac5{23^3}\right).
\]
Inserting these inequalities into \eqref{eq:two-row-penalty} gives
\begin{equation}\label{eq:low-height-penalty}
 J(x+iy;a,b)-\sqrt y\,\vartheta_4(y)\vartheta_2(y)
 \geq2\sqrt y\,e^{-5\pi y/4}(1-5e^{-3\pi y})\sin^2(\pi x).
\end{equation}
The ratio of the bound in \eqref{eq:theta-shear-loss} to that in
\eqref{eq:low-height-penalty} is at most
\[
 4e^{-3\pi y/4}\frac{(1+5e^{-3\pi y})^2}{1-5e^{-3\pi y}}
 <4\,15^{-3/4}\frac{(1+5/15^3)^2}{1-5/15^3}<\frac8{15}.
\]
This proves the comparison in this height range, strictly if $x\ne0$.

Now let $y\geq1$, so $e^{-\pi y}<1/23$. For
$1/3\leq b\leq2/3$, convexity and symmetry about $1/2$ yield
\[
 \frac1{e^{\pi yb^2}+e^{\pi y(1-b)^2}}
 \geq\frac{e^{-4\pi y/9}}{1+e^{-\pi y/3}}.
\]
Thus \eqref{eq:two-row-penalty} implies
\[
 J(x+iy;a,b)-\sqrt y\,\vartheta_4(y)\vartheta_2(y)
 \geq\frac{4\sqrt y\,e^{-13\pi y/9}(1-5e^{-3\pi y})}
 {1+e^{-\pi y/3}}\sin^2(\pi x).
\]
The ratio of the theta loss to this bound satisfies
\begin{equation}\label{eq:large-alpha-interior-bound}
\begin{aligned}
 2e^{-5\pi y/9}(1+e^{-\pi y/3})
 \frac{(1+5e^{-3\pi y})^2}{1-5e^{-3\pi y}}
 &<2\,23^{-5/9}(1+23^{-1/3})
 \frac{(1+5/23^3)^2}{1-5/23^3}
 \\
 &<\frac12<\frac8{15}.
\end{aligned}
\end{equation}
The desired comparison is therefore strict if $x\ne0$.

If $b\notin[1/3,2/3]$, monotonicity gives
\[
 \sum_ne^{-\pi y(n+b)^2}-\vartheta_2(y)
 \geq e^{-\pi y/9}+e^{-4\pi y/9}+e^{-16\pi y/9}
 -\frac{2e^{-\pi y/4}}{1-e^{-2\pi y}}.
\]
Here we retained three positive terms in the first series and used
$(n+1/2)^2\geq1/4+2n$ in the second. For
$y\geq(\log23)/\pi$,
\[
\begin{aligned}
 &\frac{36e^{-5\pi y/36}}{\pi}\frac d{dy}
 \left(e^{5\pi y/36}+e^{-7\pi y/36}+e^{-55\pi y/36}
 -\frac2{1-e^{-2\pi y}}\right)\\
 &\qquad\geq5-7e^{-\pi y/3}-55e^{-5\pi y/3}>0.
\end{aligned}
\]
Consequently,
\[
\begin{aligned}
 e^{5\pi y/36}+e^{-7\pi y/36}+e^{-55\pi y/36}
 -\frac2{1-e^{-2\pi y}}
 &\geq23^{5/36}+23^{-7/36}+23^{-55/36}
 -\frac2{1-23^{-2}}\\
 &>\frac9{100}.
\end{aligned}
\]
It follows that
\[
 J(x+iy;a,b)-\sqrt y\,\vartheta_4(y)\vartheta_2(y)
 >\frac9{100}\sqrt y\,(1-2e^{-\pi y})e^{-\pi y/4}.
\]
This penalty dominates the theta loss divided by $\alpha$, since
\[
 \frac{8e^{-7\pi y/4}(1+5e^{-3\pi y})^2}
 {(9/100)(1-2e^{-\pi y})}
 <\frac{8\,23^{-7/4}(1+5/23^3)^2}{(9/100)(1-2/23)}
 <\frac12<\frac8{15}.
\]
The comparison is strict throughout this range of $b$.

Finally, at $x=0$ the theta functions agree, while
\[
 J(iy;a,b)=\sqrt y\,\vartheta(y;a)
 \sum_ne^{-\pi y(n+b)^2}.
\]
The unique minima of its two positive factors require
$a\equiv b\equiv1/2\pmod{\mathbb Z}$.
\end{proof}

\begin{proof}[Proof of Theorem~\ref{thm:large-alpha-unrestricted}]
Modular reduction places every candidate in
$|x|\leq1/2$, $|z|\geq1$, hence $y\geq\frac{\sqrt3}{2}$.
Lemma~\ref{lem:same-height} forces every minimizer to be centered
rectangular; the reduced-domain condition then gives $y\geq1$.
The centered classification of \cite[Theorem~2.1]{LW}
now gives the asserted square and rectangular phases, since
$\alpha_b<8/15$.
\end{proof}

We now locate the first transition and combine the preceding
comparisons to prove the full classification. Concavity of the centered
minimum, together with certified endpoint signs, gives uniqueness of
its crossing with the triangular energy.\par\nopagebreak[4]

\begin{proof}[Proof of the threshold characterization]
At $\alpha=0$, the centered classification gives
\[
 \min_{\tau\in\mathbb H}E_0(\tau;1/2,1/2)
 =E_0(e^{i\pi/3};1/3,1/3)=\theta(1;e^{i\pi/3}).
\]
The centered minimum is continuous on $[0,1]$. The difference
\[
 \min_{\tau\in\mathbb H}E_\alpha(\tau;1/2,1/2)
       -E_\alpha(e^{i\pi/3};1/3,1/3)
\]
is concave: its first term is the infimum of affine functions of
$\alpha$, and its second term is affine.
Lemma~\ref{lem:cert-thresholds} gives
\begin{equation}\label{eq:crossing-endpoint-signs}
\begin{aligned}
 \min_{\tau\in\mathbb H}E_{1723/10000}(\tau;1/2,1/2)
 -E_{1723/10000}(e^{i\pi/3};1/3,1/3)
     &>0,\\
 \min_{\tau\in\mathbb H}E_{7/40}(\tau;1/2,1/2)
 -E_{7/40}(e^{i\pi/3};1/3,1/3)
     &<0.
\end{aligned}
\end{equation}
Concavity and the first inequality in
\eqref{eq:crossing-endpoint-signs} imply
\[
 \min_{\tau\in\mathbb H}E_\alpha(\tau;1/2,1/2)
 >E_\alpha(e^{i\pi/3};1/3,1/3)
 \qquad(0<\alpha\leq1723/10000).
\]
Continuity and the second inequality give a zero of the energy
difference between $1723/10000$ and $7/40$. The secant slope from
$1723/10000$ to any subsequent zero is strictly negative. Concavity
makes the difference strictly negative beyond that zero, proving
uniqueness of $\alpha_a$ and
\begin{equation}\label{eq:crossing-signs}
\begin{aligned}
 \min_{\tau\in\mathbb H}E_\alpha(\tau;1/2,1/2)
     &>E_\alpha(e^{i\pi/3};1/3,1/3)
                        &&(0<\alpha<\alpha_a),\\
 \min_{\tau\in\mathbb H}E_\alpha(\tau;1/2,1/2)
     &<E_\alpha(e^{i\pi/3};1/3,1/3)
                        &&(\alpha_a<\alpha\leq1).
\end{aligned}
\end{equation}
Throughout the initial bracket the centered minimizer is rhombic.
At the crossing, put $x_*:=\cos\varphi_{\alpha_a}$.
Its stationarity and energy equality give, with $z=x+i\sqrt{1-x^2}$,
\begin{equation}\label{eq:crossing-system}
\begin{aligned}
 E_\alpha(z;1/2,1/2)&=E_\alpha(e^{i\pi/3};1/3,1/3),\\
 \frac{d}{dx}E_\alpha(z;1/2,1/2)&=0.
\end{aligned}
\end{equation}
Conversely, a
solution of that system on the centered minimizing branch gives a
zero of the energy difference, and hence must have coupling $\alpha_a$.
Uniqueness of the centered minimizer then also determines $x_*$.

Eliminating $\alpha$ from \eqref{eq:crossing-system} gives
\begin{equation}\label{eq:crossing-residual}
 \bigl(\theta(1;z)-\theta(1;e^{i\pi/3})\bigr)
       \frac{d}{dx}J(z;1/2,1/2)
 -\frac{d}{dx}\theta(1;z)
       \bigl(J(z;1/2,1/2)-J(e^{i\pi/3};1/3,1/3)\bigr)=0.
\end{equation}
Lemma~\ref{lem:cert-thresholds} supplies a sign-changing interval
for \eqref{eq:crossing-residual}, on which the derivative quotient satisfies
\[
 -\frac{\frac{d}{dx}\theta(1;x+i\sqrt{1-x^2})}
        {\frac{d}{dx}J(x+i\sqrt{1-x^2};1/2,1/2)}
 \in\left(\frac{1723}{10000},\frac7{40}\right).
\]
The intermediate value theorem gives a zero. The centered
classification identifies its quotient with a minimizing coupling;
the single-crossing argument then identifies the pair with
$(x_*,\alpha_a)$.
The energy equality yields
\begin{equation}\label{eq:crossing-quotient}
 \alpha_a=
 \frac{\theta(1;x_*+i\sqrt{1-x_*^2})-\theta(1;e^{i\pi/3})}
 {J(e^{i\pi/3};1/3,1/3)-J(x_*+i\sqrt{1-x_*^2};1/2,1/2)}.
\end{equation}
The denominator is strictly positive, since $\alpha_a>0$ and the
unshifted theta difference is strictly positive at the nontriangular
shape $x_*+i\sqrt{1-x_*^2}$.
\end{proof}

The following lemma combines the strip, bridge, and triangular
comparisons on the rational interval containing the first transition.

\begin{lemma}[Comparison across the first transition]\label{lem:first-transition}
For $1723/10000\leq\alpha\leq7/40$, every shape and displacement satisfy
\[
 E_\alpha(z;a,b)\geq
 \min\left\{E_\alpha(e^{i\pi/3};1/3,1/3),\,
       \min_{\tau\in\mathbb H}E_\alpha(\tau;1/2,1/2)\right\}.
\]
Equality holds precisely at the triangular deep holes for
$\alpha<\alpha_a$, at the centered rhombic minimizer for
$\alpha>\alpha_a$, and at both families for $\alpha=\alpha_a$.
In the last case,
\begin{equation}\label{eq:supporting-inequality}
 E_{\alpha_a}(z;a,b)\geq E_{\alpha_a}(e^{i\pi/3};1/3,1/3)=\min_{\tau\in\mathbb H}E_{\alpha_a}(\tau;1/2,1/2),
\end{equation}
All equality statements are understood modulo the stated symmetries.
\end{lemma}

\begin{proof}
By Lemma~\ref{lem:cusp}, an unrestricted minimizer has a reduced
representative with $y<2$. We compare its energy with the smaller
of the two values in the statement.

If $x\leq39/100$, then
$y\geq\sqrt{1-(39/100)^2}>23/25$, so
Lemma~\ref{lem:centerstrip} gives
\[
 E_\alpha(z;a,b)\geq E_\alpha(z;1/2,1/2)
 \geq\min_{\tau\in\mathbb H}E_\alpha(\tau;1/2,1/2).
\]
Equality in both inequalities requires the centered displacement and,
by the centered classification, the rhombic minimizing shape.
For $x\geq39/100$, Proposition~\ref{prop:bridge} strictly excludes
every shape outside
\[
 \frac{19}{40}\leq x\leq\frac12,\qquad
 \frac{\sqrt3}{2}\leq y\leq\frac{891}{1000},
\]
since its centered competitor has energy at least the centered minimum.
This remaining cap is contained in the region of
Theorem~\ref{cor:hex-cap-comparison}, because
$891/1000<\frac{\sqrt3}{2}+1/40$. Hence
\[
 E_\alpha(z;a,b)\geq E_\alpha(e^{i\pi/3};1/3,1/3),
\]
with equality precisely at the triangular deep-hole configurations.
The unrestricted minimum is therefore at least the smaller of the
two stated energies. Both energies are attained by admissible
configurations, so this lower bound is the unrestricted minimum.
The inequality consequently holds for every shape and displacement.
Finally, \eqref{eq:crossing-signs} and equality at $\alpha_a$ give
the asserted equality configurations on each side of the transition
and at the transition itself.
\end{proof}

\begin{proof}[Completion of the proof of Theorem~\ref{thm:classification}]
Lemma~\ref{lem:first-transition} and
$\frac{1723}{10000}<\alpha_a$ give
\[
 E_{1723/10000}(z;a,b)\geq E_{1723/10000}(e^{i\pi/3};1/3,1/3),
\]
with equality precisely at the triangular deep-hole configurations.
For $0<\alpha<\frac{1723}{10000}$, affine dependence gives
\begin{equation}\label{eq:coupling-interpolation}
\begin{aligned}
 &E_\alpha(z;a,b)-E_\alpha(e^{i\pi/3};1/3,1/3)\\
 &\quad=\left(1-\frac{10000\alpha}{1723}\right)
       \bigl[\theta(1;z)-\theta(1;e^{i\pi/3})\bigr]\\
 &\qquad+\frac{10000\alpha}{1723}
       \bigl[E_{1723/10000}(z;a,b)-E_{1723/10000}(e^{i\pi/3};1/3,1/3)\bigr]
       \geq0.
\end{aligned}
\end{equation}
Both coefficients are strictly positive, and both bracketed differences
are nonnegative by Montgomery's theorem and the comparison at $\frac{1723}{10000}$.
Equality forces the triangular deep-hole configurations. Thus the
entire interval $0<\alpha<\frac{1723}{10000}$ follows from a single comparison
at a rational coupling.

The remaining comparisons cover
\[
 [\frac{1723}{10000},7/40]\cup[7/40,8/15]\cup[8/15,1]=[\frac{1723}{10000},1].
\]
On $[\frac{1723}{10000},7/40]$, Lemma~\ref{lem:first-transition} and
\eqref{eq:crossing-signs} give the triangular phase up to $\alpha_a$,
both minimizing families at $\alpha_a$, and the centered rhombic
phase above it. We use Theorem~\ref{thm:centered} only on
$[7/40,8/15]$, which is contained in its stated range, and
Theorem~\ref{thm:large-alpha-unrestricted} on $[8/15,1]$.
Every minimizer for $\alpha>\alpha_a$ is therefore centered.
The classification of \cite[Theorem~2.1]{LW} gives the
rhombic, square, and rectangular phases, their transition endpoints,
and $y_1=\sqrt3$. This completes the positive-coupling classification.
At $\alpha=0$, Montgomery's theorem fixes the triangular shape,
while $E_0=\theta(1;z)$ is independent of displacement.

For completeness, Lemma~\ref{lem:analytic-properties} also gives
the nonpositive-coupling cases. For $-1<\alpha<0$,
\[
 E_\alpha(z;a,b)\geq(1+\alpha)\theta(1;z)
 \geq(1+\alpha)\theta(1;e^{i\pi/3}),
\]
with equality exactly for a triangular lattice and coincident
displacement. At $\alpha=-1$, the energy is $\theta(1;z)-J(z;a,b)$,
so every coincident configuration minimizes, and these are the only
equality configurations.
\end{proof}

\appendix
\section{Approximations and finite-sum estimates}\label{app:certificates}

With $Q_{m,n}$ as in \eqref{eq:lattice-quadratic-form}, put
\[
\begin{aligned}
 \theta_4(z)&=\sum_{|m|,|n|\le4}e^{-\pi Q_{m,n}},\\
 J_4(z;a,b)&=\sum_{|m|,|n|\le4}e^{-\pi Q_{m,n}}
                           \cos\bigl(2\pi(ma+nb)\bigr).
\end{aligned}
\]
A square ring of radius $r=\max(|m|,|n|)$ contains $8r$ points.
We repeatedly use, for $k\ge0$ and $(6/5)^ke^{-11\pi c}<1$,
\begin{equation}\label{eq:common-ring-tail}
 \sum_{r\ge5}r^ke^{-\pi cr^2}
 \le\frac{5^ke^{-25\pi c}}{1-(6/5)^ke^{-11\pi c}}.
\end{equation}
Indeed, $r=5+j$ gives $r^2\ge25+11j$ and
$r^k\le5^k(6/5)^{kj}$. In particular, if $Q_{m,n}\ge cr^2$,
\begin{equation}\label{eq:common-value-tail}
 |\theta(1;z)-\theta_4(z)|,\ |J-J_4|
 \le8\sum_{r\ge5}re^{-\pi cr^2}.
\end{equation}
 These
approximations, together with \eqref{eq:common-ring-tail}, reduce
all estimates below to finite inequalities.

\begin{lemma}[Centered-strip estimate]\label{lem:cert-strip}
With $A,B$ as in Lemma~\ref{lem:four} and $Q_{m,n}$ as in
\eqref{eq:lattice-quadratic-form}, define
\[
 H(x,y)=2\begin{pmatrix}A&0\\0&B\end{pmatrix}
 -\sum_{\substack{m,n\in\mathbb Z\\m+n\ \mathrm{even}}}
 e^{-\pi Q_{m,n}}\begin{pmatrix}m^2&|mn|\\|mn|&n^2\end{pmatrix}.
\]
If $0\leq x\leq39/100$ and $23/25\leq y\leq2$, then
\[
 H_{11}>\frac3{5000},\qquad H_{22}>\frac{17}{500},\qquad
 \det H>\frac1{10000},\qquad H\succ0.
\]
\end{lemma}

\begin{proof}
The determinant-one matrix of $Q_{m,n}$ has trace satisfying
\[
 y+\frac{1+x^2}{y}
 \le\frac{51521}{20000}<\frac{47}{100}+\frac{100}{47}.
\]
Hence $Q_{m,n}\ge(47/100)(m^2+n^2)$. Each omitted weighted
sum is bounded by
\[
 8\sum_{r\ge5}r^3e^{-47\pi r^2/100}<2\cdot10^{-12}.
\]
Use the $64$ rectangles
\[
 K_{jk}=\left[\frac{39j}{400},\frac{39(j+1)}{400}\right]
 \times\left[\frac{368+27k}{400},\frac{395+27k}{400}\right],
 \quad0\le j<4,\quad0\le k<16.
\]
Let $\Lambda_4$ contain the pairs $0\le m\le4$, $|n|\le4$,
$m+n$ even, with $m>0$ or $n>0$. On each $K=K_{jk}$, set
\[
\begin{aligned}
 u_K&=\sum_{\substack{(m,n)\in\Lambda_4\\m\ne0}}
 m^2e^{-\pi\inf_K(Q_{m,n}-Q_{1,0})}+\frac{10^{-12}}{\inf_KA},\\
 v_K&=\sum_{\substack{(m,n)\in\Lambda_4\\n\ne0}}
 n^2e^{-\pi\inf_K(Q_{m,n}-Q_{0,1})}+\frac{10^{-12}}{\inf_KB},\\
 w_K&=\sum_{\substack{(m,n)\in\Lambda_4\\mn\ne0}}|mn|
 e^{-\pi\inf_K\{Q_{m,n}-(Q_{1,0}+Q_{0,1})/2\}}
 +\frac{10^{-12}}{\sqrt{(\inf_KA)(\inf_KB)}}.
\end{aligned}
\]
All infima here have the form
$\inf_K\{ay+(ux^2-2vx+w)/y\}$, with $a,u\ge0$:
minimize the quadratic over the $x$ interval, then test the two
$y$ endpoints and $\sqrt{(ux^2-2vx+w)/a}$ when it belongs to
the $y$ interval and its numerator and $a$ are positive. Convexity gives $\inf_K A$ and $\inf_K B$
from the corner maxima of their exponents. The resulting finite
inequalities on every $K$ are
\[
\begin{aligned}
 2(\inf_KA)(1-u_K)&>\frac3{5000},\\
 2(\inf_KB)(1-v_K)&>\frac{17}{500},\\
 4(\inf_KA)(\inf_KB)\{(1-u_K)(1-v_K)-w_K^2\}
 &>\frac1{10000}.
\end{aligned}
\]
These bound $H_{11}$, $H_{22}$, and $\det H$, respectively.
Sylvester's criterion completes the proof.
\end{proof}

The next lemma gives the centered competitors and the global energy gaps.

\begin{lemma}[Energy comparisons]\label{lem:cert-global}
Let
\[
\begin{aligned}
 (\mathcal I_1,\mathcal I_2,\mathcal I_3)&=([7/40,9/50],\,[9/50,1/5],\,[1/5,6/25]),\\
 (\mathcal I_4,\mathcal I_5,\mathcal I_6)&=([6/25,3/10],\,[3/10,37/100],\,[37/100,2/5]),\\
 (t_1,\ldots,t_6)&=(371/1000,\,9/25,\,33/100,\,27/100,\,17/100,\,0),
\end{aligned}
\]
and put $z_j=t_j+i\sqrt{1-t_j^2}$. For $1\leq j\leq6$,
$\alpha\in \mathcal I_j$, every reduced shape with $39/100\leq x\leq1/2$,
$y\leq2$, and every displacement satisfy
\[
 E_\alpha(x+iy;a,b)-E_\alpha(z_j;1/2,1/2)>3\cdot10^{-8}.
\]
For $2/5\leq\alpha\leq23/25$, every reduced shape with
$3/10\leq x\leq1/2$, $y\leq2$, and every displacement satisfy
\[
 E_\alpha(x+iy;a,b)-E_\alpha(i;1/2,1/2)>5\cdot10^{-6}.
\]
For $1723/10000\leq\alpha\leq7/40$, the energy difference
\[
 E_\alpha(x+iy;a,b)
 -E_\alpha\left(\frac{3+i\sqrt{55}}8;\frac12,\frac12\right)
\]
is greater than $1.9\cdot10^{-8}$ when
$39/100\leq x\leq19/40$, $y\leq2$, and greater than
$2.4\cdot10^{-8}$ when $19/40\leq x\leq1/2$,
$891/1000\leq y\leq2$, for every reduced shape in the
indicated regions and every displacement.
\end{lemma}

\begin{proof}
Let $f_\alpha$ be the lower bound in \eqref{eq:lower}; for the
square comparison omit its finite remainder. For a rectangle
$K=[x_-,x_+]\times[y_-,y_+]$, first discard it if
$x_+^2+y_+^2<1$, and otherwise replace $y_-$ by
$\max\{y_-,\sqrt{1-x_+^2}\}$. The clipped rectangle contains
$K\cap\{x^2+y^2\ge1\}$. Convexity gives
\[
 c^-_{m,n}(K)=\exp\!\left[-\pi
 \max_{\substack{x\in\{x_-,x_+\}\\y\in\{y_-,y_+\}}}Q_{m,n}(x,y)\right].
\]
The retained energy is increasing in every coefficient, since
$1+\alpha\cos(2\pi(ma+nb))\ge0$. Substitution of the
$c^-_{m,n}(K)$ in \eqref{eq:lower} therefore gives a lower bound
$L_\alpha(K)$ for $f_\alpha$. Lemma~\ref{lem:four} also gives
\[
 M_\alpha(K)=\underline{f_\alpha(x_c,y_c)}
 -r_x\sup_K|\partial_xf_\alpha|-r_y\sup_K|\partial_yf_\alpha|
 \le\inf_Kf_\alpha,
\]
where $(x_c,y_c)$ and $(r_x,r_y)$ are its center and half-widths.
The local Lipschitz property justifies this estimate across branch
changes. A test is omitted if its coefficient, denominator, or
radicand conditions are not yet verified.

At each coupling endpoint, start with
\[
\begin{array}{ll}
 [39/100,1/2]\times[43/50,2]&\text{for the six arc comparisons},\\
 {}[3/10,1/2]\times[43/50,2]&\text{for the square comparison},\\
 {}[39/100,19/40]\times[43/50,2]&\text{for the first bridge},\\
 {}[19/40,1/2]\times[891/1000,2]&\text{for the second bridge}.
\end{array}
\]
Accept a clipped rectangle precisely when
\begin{equation}\label{eq:energy-box-test}
 \max\{L_\alpha(K),M_\alpha(K)\}
 -\overline{E_\alpha(z_t;1/2,1/2)}>\delta,
 \qquad z_t=t+i\sqrt{1-t^2},
\end{equation}
where $\delta$ is the asserted gap. Bisect each remaining rectangle
along its longer side, taking the $x$ side in a tie, and repeat.
The tests terminate by depth $15$. In the four groups above, the
smallest accepted lower bounds exceed, respectively,
\[
 \frac1{4000000},\qquad\frac7{125000},\qquad
 \frac{11}{5000000},\qquad\frac{11}{10000000}.
\]
Thus \eqref{eq:energy-box-test} holds on a finite cover of each
reduced region. Only $|m|,|n|\le2$ are retained in the positive
remainder of \eqref{eq:lower}; all omitted terms are nonnegative.
The competitor upper bounds use $\theta_4+\alpha J_4$ and
\[
 16\sum_{r\ge5}re^{-9\pi r^2/10}<10^{-27},
\]
since the competitor heights exceed $9/10$ and
$Q_{m,n}\ge y\max(m^2,n^2)$ on their unit-circle shapes.
Finally, the energy difference from a fixed competitor is affine in
$\alpha$. Its endpoint bounds imply each assertion on the entire
coupling interval.
\end{proof}

The next lemma uses $\mathcal R$, $\operatorname{osc}_{a,b}$,
$\nabla_{a,b}$, and the Euclidean norm $\|\cdot\|_2$ defined in the
proof of Lemma~\ref{lem:hex-shift-localization}.

\begin{lemma}[Displacement localization]\label{lem:cert-localization}
Let $19/40\leq x\leq1/2$ and
$0\leq y-\frac{\sqrt3}{2}\leq1/40$. Then the coefficient bounds
\eqref{eq:hex-cap-coarse} and the affine enclosures
\eqref{eq:hex-cap-affine-centers} hold, and
\[
 \operatorname{osc}_{a,b}\mathcal R<\frac{131}{500000},\qquad
 \sup_{(a,b)\in[0,1]^2}\|\nabla_{a,b}\mathcal R(a,b)\|_2
 <\frac{39}{25000}.
\]
\end{lemma}

\begin{proof}
The phasor lengths $P,Q,R$ defined in the proof of
Lemma~\ref{lem:hex-shift-localization} have the exact expressions
\[
 P=e^{-\pi x/y},\qquad
 Q=e^{-\pi\{y+(x^2-x)/y\}},\qquad R=e^{-\pi(1-x)/y}.
\]
Corner bounds and one-variable extrema prove the coefficient
bounds in \eqref{eq:hex-cap-coarse}. For its cosine bounds, use
\eqref{eq:hex-phasor-centers} on the rectangles
\begin{equation}\label{eq:localization-cover}
 K_{jk}=\left[\frac{19}{40}+\frac j{1280},
                 \frac{19}{40}+\frac{j+1}{1280}\right]
 \times\left[\frac{\sqrt3}{2}+\frac k{1280},
                 \frac{\sqrt3}{2}+\frac{k+1}{1280}\right],
 \quad0\le j,k<32.
\end{equation}
The first two cosine ranges lie in $(-1,1)$, which also proves
the strict triangle inequalities.

Let $\Lambda'_4$ consist of $0\le m\le4$, $|n|\le4$,
$m>0$ or $n>0$, with $(1,0),(0,1),(1,1)$ removed.
The determinant-one quadratic form has trace less than $231/100$;
thus $Q_{m,n}\ge(57/100)(m^2+n^2)$. Bisecting both cap intervals
once gives four rectangles $K$. On each of them,
\[
\begin{aligned}
 \operatorname{osc}_{a,b}\mathcal R
 &\le4\sum_{(m,n)\in\Lambda'_4}e^{-\pi\inf_KQ_{m,n}}
       +2\cdot10^{-16}<\frac{131}{500000},\\
 \sup_{a,b}\|\nabla_{a,b}\mathcal R\|_2
 &\le4\pi\sum_{(m,n)\in\Lambda'_4}\sqrt{m^2+n^2}
        e^{-\pi\inf_KQ_{m,n}}+2\cdot10^{-14}<\frac{39}{25000}.
\end{aligned}
\]
The two error terms follow from \eqref{eq:common-ring-tail} with
$(c,k)=(57/100,1)$ and $(57/100,2)$. The infima are evaluated
as in Lemma~\ref{lem:cert-strip}.

For the affine enclosures, apply the box mean-value estimate on
\eqref{eq:localization-cover} to
\[
\begin{aligned}
 F_1&=\cos\!\left(2\pi\left[\frac13-\left(x-\frac12\right)
 -\frac{29}{50}\left(y-\frac{\sqrt3}{2}\right)-\frac1{200}\right]\right)-\cos U_0,\\
 F_2&=\cos U_0-\cos\!\left(2\pi\left[\frac13-\left(x-\frac12\right)
 -\frac{29}{50}\left(y-\frac{\sqrt3}{2}\right)+\frac1{200}\right]\right),\\
 F_3&=\cos\!\left(2\pi\left[\frac13+\frac{29}{25}
 \left(y-\frac{\sqrt3}{2}\right)-\frac7{2000}\right]\right)-\cos V_0,\\
 F_4&=\cos V_0-\cos\!\left(2\pi\left[\frac13+\frac{29}{25}
 \left(y-\frac{\sqrt3}{2}\right)+\frac7{2000}\right]\right).
\end{aligned}
\]
The finite inequalities on every rectangle, with center $c_K$, are
\[
 \underline{F_i(c_K)}-\frac1{2560}
   \bigl(\sup_K|\partial_xF_i|+\sup_K|\partial_yF_i|\bigr)
 >\frac1{20000},\qquad1\le i\le4.
\]
All four affine endpoints lie in $(0,1/2)$. Monotonicity of cosine
on $[0,\pi]$ proves \eqref{eq:hex-cap-affine-centers}.
\end{proof}

The next lemma certifies displacement control and joint convexity.

\begin{lemma}[Hessian estimates]\label{lem:cert-hessians}
Throughout $19/40\leq x\leq1/2$,
$0\leq y-\frac{\sqrt3}{2}\leq1/40$,
\[
 D^2_{a,b}J(x+iy;a,b)\geq\frac{13}{20}I_2
 \qquad\text{if }|a-a_0|,|b-b_0|\leq\frac3{100}.
\]
The full Hessian of $E_\alpha$ in $(x,y,a,b)$ is positive definite at
$\alpha=4/25$ on the affine prism
\eqref{eq:hex-affine-prism} enlarged to
$|r_a|,|r_b|\leq1/100$, and at $\alpha=7/40$ on
\eqref{eq:hex-affine-prism} itself.
\end{lemma}

\begin{proof}
The displacement Hessian is
\[
 D^2_{a,b}J=-4\pi^2\sum_{m,n\in\mathbb Z}e^{-\pi Q_{m,n}}
 \cos\bigl(2\pi(ma+nb)\bigr)
 \begin{pmatrix}m^2&mn\\mn&n^2\end{pmatrix}.
\]
For the joint Hessian use \eqref{eq:hex-joint-hessian}.
On the rational outer shape rectangle
\[
 \frac{4749}{10000}\le x\le\frac{5001}{10000},\qquad
 \frac{433}{500}\le y\le\frac{8911}{10000},
\]
completing the square gives $Q_{m,n}\ge(17/20)r^2$.
The first and second shape derivatives of $Q_{m,n}$ are bounded
by $5r^2$ and $8r^2$. Every summand Hessian entry is therefore
bounded by $900r^4e^{-\pi Q_{m,n}}$. By
\eqref{eq:common-ring-tail}, the omitted operator norm satisfies
\begin{equation}\label{eq:hex-joint-hessian-tail}
 28800\sum_{r\ge5}r^5e^{-51r^2/20}<2\cdot10^{-20}.
\end{equation}
Truncate at $|m|,|n|\le4$ and enlarge every entry by
$[-9\cdot10^{-13},9\cdot10^{-13}]$.

Parameterize the joint prisms by
\[
 x=\frac{19+s}{40},\qquad y=\frac{\sqrt3}{2}+\frac t{40},\qquad
 r_a=\rho_a(2u-1),\quad r_b=\rho_b(2v-1),\qquad
 (s,t,u,v)\in[0,1]^4,
\]
with $(\rho_a,\rho_b)=(1/100,1/100)$ at $\alpha=4/25$ and
$(3/400,3/500)$ at $\alpha=7/40$; $a,b$ are given by
\eqref{eq:hex-affine-prism}. For the displacement estimate use
$a=a_0+3(2u-1)/100$ and $b=b_0+3(2v-1)/100$ instead.
The centers are enclosed by \eqref{eq:hex-phasor-centers} and
directed cosine tests on $(0,\pi)$.

Set $D=D^2_{a,b}J-(13/20)I_2$. Starting with $[0,1]^4$,
bisect a failing box of depth $d$ in coordinate $1+(d\bmod4)$.
For $D$, accept a box when
\[
 \underline{D_{11}}>0,\qquad
 \underline{D_{11}D_{22}-D_{12}^2}>0.
\]
For the joint matrix $H=D^2_{x,y,a,b}E_\alpha$, let $M$ be the
entrywise midpoint of its interval enclosure. Write its spectral
decomposition as
\[
 M=V\operatorname{diag}(\lambda_1,\ldots,\lambda_4)V^{\mathsf T},
 \qquad V^{\mathsf T}V=I_4,
\]
where $\lambda_i$ are the eigenvalues of $M$. If $\lambda_i>0$
for $1\le i\le4$, round the entries of
$V\operatorname{diag}(\lambda_i^{-1/2})$ to multiples of $2^{-20}$
to obtain a fixed matrix $P$. Accept the box when
\begin{equation}\label{eq:hessian-box-test}
 \underline{(P^{\mathsf T}HP)_{ii}}
 -\sum_{j\ne i}\max\bigl\{
 |\underline{(P^{\mathsf T}HP)_{ij}}|,
 |\overline{(P^{\mathsf T}HP)_{ij}}|\bigr\}>0,
 \qquad1\le i\le4.
\end{equation}
Only the exact dyadic entries of $P$ enter this test; the accuracy
of its construction is immaterial to the implication.
The three subdivisions terminate within $31$ cyclic bisections,
and all acceptance margins exceed $10^{-8}$.
All inequalities include the entrywise enlargement above.
Sylvester's criterion proves $D\succ0$. Strict diagonal dominance
proves $P^{\mathsf T}HP\succ0$, hence $H\succ0$, on every box.
The completes the proof.
\end{proof}

Finally, we give the interval information that locates the transitions.
The first-crossing tests concern the actual centered minima, whose
uniqueness is supplied by the our previous result \cite{LW}.

\begin{lemma}[Transition enclosures]\label{lem:cert-thresholds}
The ratios \eqref{eq:centered-thresholds} satisfy
\[
 0.373<\alpha_b<0.374,\qquad0.925<\alpha_c<0.926,
\]
and round to $0.3732155079$ and $0.9256496974$, respectively,
to ten decimal places. The centered energy satisfies
\[
\begin{aligned}
 \min_{\tau\in\mathbb H}E_{1723/10000}(\tau;1/2,1/2)
 -E_{1723/10000}(e^{i\pi/3};1/3,1/3)
     &>10^{-6},\\
 \min_{\tau\in\mathbb H}E_{7/40}(\tau;1/2,1/2)
 -E_{7/40}(e^{i\pi/3};1/3,1/3)
     &<-7\cdot10^{-5}.
\end{aligned}
\]
The left-hand side of \eqref{eq:crossing-residual} is negative
and positive, respectively, at
\[
 x=0.3754117719806,\qquad x=0.3754117719808.
\]
For every $x$ between these endpoints, the quotient
\[
 -\frac{\frac{d}{dx}\theta(1;x+i\sqrt{1-x^2})}
        {\frac{d}{dx}J(x+i\sqrt{1-x^2};1/2,1/2)}
\]
is well defined and lies in
$(0.1723565872424312,0.1723565872443172)$.
\end{lemma}

\begin{proof}
At $z=i$, the summands of $\theta_{xx}$ and $\theta_{yy}$ are
\[
 e^{-\pi(m^2+n^2)}(4\pi^2m^2n^2-2\pi m^2),\qquad
 e^{-\pi(m^2+n^2)}\{\pi^2(m^2-n^2)^2-2\pi n^2\}.
\]
For the centered $J$ derivatives insert the factor $(-1)^{m+n}$.
Each omitted derivative tail is bounded by
\[
 640\sum_{r\ge5}r^5e^{-3r^2}<6\cdot10^{-27}.
\]
Both denominator enclosures exclude zero. Interval division in
\eqref{eq:centered-thresholds} proves the stated enclosures and rounding.

On the arc $z=x+i\sqrt{1-x^2}$, primes denote derivatives of
$x\mapsto Q_{m,n}(x,\sqrt{1-x^2})$, with $Q_{m,n}$ from \eqref{eq:lattice-quadratic-form}.
Thus
\[
\begin{aligned}
 Q_{m,n}&=\frac{m^2+n^2-2mnx}{\sqrt{1-x^2}},\\
 Q_{m,n}'&=-\frac{2mn}{\sqrt{1-x^2}}
     +\frac{x(m^2+n^2-2mnx)}{(1-x^2)^{3/2}},\\
 Q_{m,n}''&=\frac{-4mnx(1-x^2)+(1+2x^2)(m^2+n^2-2mnx)}
 {(1-x^2)^{5/2}}.
\end{aligned}
\]
For $3/10\le x\le2/5$, we have
\[
 Q_{m,n}\ge\frac9{10}\max\{|m|,|n|\}^2,\qquad
 |Q_{m,n}'|<10\max\{|m|,|n|\}^2,\qquad
 |Q_{m,n}''|<30\max\{|m|,|n|\}^2
\]
for $(m,n)\ne(0,0)$.
Consequently the value, first-derivative, and second-derivative
errors in each truncated theta sum are bounded by
\[
 10^{-27},\qquad10^{-22},\qquad10^{-19},
\]
respectively, by \eqref{eq:common-ring-tail}. At the triangular
shape, the triangular deep-hole factors are
\[
 \cos\frac{2\pi(m+n)}3=
 \begin{cases}1,&m+n\equiv0\pmod3,\\-1/2,&m+n\not\equiv0\pmod3,
 \end{cases}
\]
and the value error is below $10^{-25}$.

At $\alpha=1723/10000,7/40$, the arc derivative is negative
at $x=3/10$ and positive at $x=2/5$. Forty-two rational
bisections enclose each stationary point in an interval of length
$1/(10\cdot2^{42})$. The Luo--Wei classification identifies these
points as the centered minima. Finite-sum evaluation on these
intervals proves the two energy inequalities in the statement.

Finally, substitution in the left-hand side of
\eqref{eq:crossing-residual} gives values below $-5\cdot10^{-16}$
and above $10^{-15}$ at
$x=0.3754117719806$ and $x=0.3754117719808$, respectively.
Moreover,
\[
 \frac d{dx}J(x+i\sqrt{1-x^2};1/2,1/2)>\frac{17}{50}
 \qquad(0.3754117719806\le x\le0.3754117719808).
\]
On the same interval,
\[
 \left|\frac d{dx}\left[
 -\frac{\frac d{dx}\theta(1;x+i\sqrt{1-x^2})}
 {\frac d{dx}J(x+i\sqrt{1-x^2};1/2,1/2)}
 \right]\right|<\frac65.
\]
Midpoint evaluation and this derivative bound give the asserted
quotient enclosure. This proves the lemma.
\end{proof}

\medskip

\section*{\bf Acknowledgements:} The research of S. Luo is partially supported by the National Natural Science Foundation
of China (NSFC) under Grant Nos. 12261045 and 12001253, and by the Jiangxi Jieqing
Fund under Grant No. 20242BAB23001. The research of J. Wei is partially supported by the General Research Fund (GRF) of Hong Kong "On Fujita equation in critical or supercritical regime". The authors conducted all mathematical arguments, proofs, and computations in this manuscript.

\enlargethispage{3\baselineskip}
\end{document}